\documentclass[11pt]{amsart}
\usepackage{amsmath,amssymb,latexsym,soul,cite}
\usepackage{color,enumitem,graphicx}
\usepackage[colorlinks=true,urlcolor=blue,
citecolor=red,linkcolor=blue,linktocpage,pdfpagelabels,
bookmarksnumbered,bookmarksopen]{hyperref}
\usepackage[english]{babel}
\usepackage[T1]{fontenc}
\usepackage[hyperpageref]{backref}

\usepackage[left=2.8cm,right=2.8cm,top=2.9cm,bottom=2.9cm]{geometry}

\numberwithin{equation}{section}

\newtheorem{thm}{Theorem}[section]
  \theoremstyle{plain}
  \newtheorem{lem}[thm]{Lemma}
  \theoremstyle{plain}
  \newtheorem{prop}[thm]{Proposition}
  \theoremstyle{plain}
  \newtheorem{cor}[thm]{Corollary}
  \theoremstyle{plain}
  \newtheorem{defn}[thm]{Definition}
  \theoremstyle{definition}
\newtheorem{remark}[thm]{Remark}

\newcommand{\R}{\mathbb{{R}}}

\newcommand{\N}{\mathbb{{N}}}

\newcommand{\Ne}{\mathcal{N}}

\newcommand{\Lp}{L^p}
\newcommand{\Lq}{L^q}

\newcommand{\Ho}{H_{0}^{1}(\Omega)}
\newcommand{\Hoo}{H_{0}^{2}(\Omega)}
\newcommand{\Omegabar}{\overline{\Omega}}
\newcommand{\ubar}{\overline u}
\newcommand{\Holder}{\text{H\"{o}lder}}

\newcommand{\Frechet}{\text{Fr\'{e}chet}}

\newcommand{\dOmega}{{\partial\Omega}}
\newcommand{\dB}{{\partial B}}

\title[]{Positivity for fourth-order semilinear problems related to the Kirchhoff-Love functional}

\author[]{Giulio Romani*}

\address[Giulio Romani]{Aix-Marseille Universit\'e, CNRS
	\newline\indent
	Centrale Marseille, I2M, UMR 7373, 39 Rue Frederic Joliot
	\newline\indent
	Curie, 13453 Marseille, France}
\email{giulio.romani@univ-amu.fr, giulio.romani@unimi.it}

\thanks{*This work has been carried out thanks to the support of the A*MIDEX grant (n.ANR-11-IDEX-0001-02) funded by the French Government "Investissements d'Avenir" program.}
\subjclass[2010]{}
\keywords{}

\begin{document}

\begin{abstract}
	We study the ground states of the following generalization of the Kirchhoff-Love functional,
	$$J_\sigma(u)=\int_\Omega\dfrac{(\Delta u)^2}{2} - (1-\sigma)\int_\Omega det(\nabla^2u)-\int_\Omega F(x,u),$$
	where $\Omega$ is a bounded convex domain in $\R^2$ with $C^{1,1}$ boundary and the nonlinearities involved are of sublinear type or superlinear with power growth. These critical points correspond to least-energy weak solutions to a fourth-order semilinear boundary value problem with Steklov boundary conditions depending on $\sigma$. Positivity of ground states is proved with different techniques according to the range of the parameter $\sigma\in\R$ and we also provide a convergence analysis for the ground states with respect to $\sigma$. Further results concerning positive radial solutions are established when the domain is a ball.
\end{abstract}

\maketitle

\section{Introduction}
The energy of a thin hinged plate under the action of a vertical external force of density $f$ can be computed by the Kirchhoff-Love functional
\begin{equation*}\label{I}
	I_\sigma(u)=\int_\Omega\dfrac{(\Delta u)^2}{2} - (1-\sigma)\int_\Omega det(\nabla^2u)-\int_\Omega fu,
\end{equation*}
where the bounded domain $\Omega\subset\R^2$ describes the shape of the plate and $u$ its deflection from the original unloaded position. Since the plate is supposed to be hinged, the natural space in which to consider our problem is $H^2(\Omega)\cap\Ho$. The coefficient $\sigma$, called \textit{Poisson ratio}, depends on the material and measures its transverse expansion (resp. contraction), according to its positive (resp. negative) sign, when subjected to an external compressing force. Due to some thermodynamic considerations in elasticity theory, the physical relevant interval for $\sigma$ is $(-1,\frac{1}{2})$. A detailed derivation of the model can be found in \cite{VK}, while a mathematical analysis concerning the positivity preserving property for $I_\sigma$ has been carried out by Parini and Stylianou in \cite{PS}. Besides a further extension of their results, here we are interested in a direct generalization of the Kirchhoff-Love functional, namely when the density $f$ may depend also on the deflection of the plate itself:
\begin{equation}\label{J}
J_\sigma(u)=\int_\Omega\dfrac{(\Delta u)^2}{2} - (1-\sigma)\int_\Omega det(\nabla^2u)-\int_\Omega F(x,u),
\end{equation}
where $F(x,s)=\int_0^sf(x,t)dt$, and furthermore we let $\sigma\in\R$.  We are mainly interested in a power-type nonlinearity, namely
\begin{equation}\label{FNL}
F(x,u)=\dfrac{g(x)|u|^{p+1}}{p+1},\quad\mbox{where}\;\,g\in L^1(\Omega)\;\,\mbox{and}\;\, g>0\;\, \mbox{in}\,\,\Omega.
\end{equation}
In particular we look for existence and positivity of those critical points which have the lowest energy, referred in the literature as \textit{ground states}.\\
\noindent If the boundary is sufficiently smooth, searching critical points of $J_\sigma$ with the nonlinearity \eqref{FNL} is equivalent to find weak solutions of the following fourth-order semilinear boundary problem
	\begin{equation}\label{PDEpSteklovg}
	\begin{cases}
	\Delta^2u=g(x)|u|^{p-1}u\quad&\mbox{in }\Omega,\\
	u=\Delta u-(1-\sigma)\kappa u_n=0\quad&\mbox{on }\dOmega,
	\end{cases}
	\end{equation}
where $u_n$ stands for the normal derivative of $u$ on $\dOmega$ and $\kappa$ is the signed curvature of the boundary (positive on convex parts). This kind of mixed boundary conditions are usually called \textit{Steklov} from their first appearance in \cite{Stek} 
and they are an intermediate situation between Navier boundary conditions (when $\sigma=1$) and Dirichlet boundary conditions ($u=u_n=0$, seen as the limit case as $\sigma\rightarrow+\infty$).\\
\noindent Although fourth-order (or more generally, higher-order) problems have arisen attention even from the first decade of the 20th century, most of the literature deals with the Navier case, where the maximum principle still holds, or with Dirichlet boundary conditions, where Green's function arguments are available. Conversely, problems like \eqref{PDEpSteklovg} have been intensively studied only in the last decade, focusing on the associated boundary eigenvalue problems (see \cite{FGW2,BucurFG}), the positivity preserving property of the solution operator (see \cite{GS}) and some semilinear problems (for instance, \cite{BG,BGM,BGW}).\\
\noindent This paper is a contribution to the study of semilinear subcritical biharmonic Steklov problems in low dimension. Here, we mainly focus on a nonlinearity of power-type as in \cite{BGW}, where the critical exponent in high dimensions is considered and the domain is a ball. On the other side, although some related subcritical problems have already appeared in \cite{BGM}, we consider slightly different kind of nonlinearity, we let $\sigma$ be not only lying in the physical relevant interval and the techniques involved are different.\\
\noindent Besides the existence of ground states for $J_\sigma$, we mainly investigate their positivity. The question is quite challenging since, like most fourth-order problems, one has to face the lack of a maximum principle. Moreover, we will show that positivity is strongly related to the parameter $\sigma$ and different techniques are needed to cover different regions in which $\sigma$ lies: the superharmonic method, some convergence arguments and the dual cones decomposition.\\
\noindent The main results contained in this paper may be summarized as follows:
\begin{thm}[Existence, Positivity]\label{thm1}
Let $\Omega\subset\R^2$ be a bounded convex domain with $\dOmega$ of class $C^{1,1}$ and let $f(x,s)=g(x)|s|^{p-1}s$, with $p\in(0,1)\cup(1,+\infty)$ and $g\in L^1(\Omega)$, $g>0$ a.e. in $\Omega$. Then there exist $\sigma^*\leq-1$ and $\sigma_1>1$ (depending on $\Omega$ and possibly infinite) such that the functional $J_\sigma$ has no positive critical points if $\sigma\leq\sigma^*$, while it admits (at least) a positive ground state if $\sigma\in(\sigma^*,\sigma_1)$.
\end{thm}
\begin{thm}[Convergence]\label{thm2}
Under the previous assumptions for $\Omega$ and $f$, let $(u_k)_{k\in\N}$ be a sequence of ground states for the respective sequence of functionals $(J_{\sigma_k})_{k\in\N}$. Up to a subsequence,
\begin{enumerate}[label=(\roman*)]
	\item if $\sigma_k\searrow\sigma^*$, then $u_k\rightarrow0$ in $H^2(\Omega)$ in case $p>1$, while $u_k\rightarrow+\infty$ in $L^\infty(\Omega)$ if $p\in(0,1)$;
	\item if $\sigma_k\rightarrow1$, then $u_k\rightarrow\ubar$ in $W^{2,q}(\Omega)$ for every $q>2$, where $\ubar$ is a ground state for the Navier problem;
	\item if $\sigma_k\rightarrow+\infty$, then $u_k\rightarrow U$ in $H^2(\Omega)$, where $U$ is a ground state for the Dirichlet problem.
\end{enumerate}
\end{thm}
\noindent Notice that Theorem \ref{thm1} might also be seen as an extension to the semilinear setting of the main positivity results established by Gazzola and Sweers in \cite[Theorem 4.1]{GS} for the linear case.\\
\noindent Finally, we want to stress our attempt to impose only the strictly necessary assumptions on the domain in order to obtain our results and to have a well-defined second boundary condition in \eqref{PDEpSteklovg}.
\vskip0.2truecm
 The paper is organized as follows: after a few preliminary results (Section 2), we establish existence (Section 3) and positivity (Section 4) of ground states of $J_\sigma$ when $\sigma$ belongs to the range $(-1,1]$ (which contains the relevant physical interval) both for $f$ sublinear and superlinear; the latter is due to an argument based on the Nehari manifold. Except for the last section, the rest of the paper is devoted to complete Theorem \ref{thm1}
in the cases $\sigma\leq-1$ (Section 5) and $\sigma>1$ (Section 6). While the first situation is quite easy to handle, the positivity in the second is more delicate and requires different tools. In this context and also for this purpose, Theorem \ref{thm2} will be established. Finally, Section 7 provides a further investigation in case $\Omega$ is the unit ball, concerning generic positive radially symmetric solutions.
\vskip0.4truecm
\textbf{Acknowledgements}: The author wishes to express his gratitude to Enea Parini for his active interest in the publication of this paper and to Fran\c{c}ois Hamel for many valuable and stimulating conversations. The author wants also to thank the anonymous referee for the careful reading of the manuscript and helpful comments and suggestions.

\section{Notation and preliminary results}

Throughout the paper, $\nabla^2u$ stands for the Hessian matrix of $u$ and the derivatives are denoted by subscripts ($u_x$, $u_{xy}$, ...). Moreover, $n$ and $\tau$ will be the exterior normal and the tangent vector, and $u_n$ and $u_\tau$ the normal and the tangential derivative of $u$. We say that $u$ is \textit{superharmonic} in $\Omega$ when $-\Delta u\geq 0$ in $\Omega$ and $u=0$ on $\dOmega$; $u$ is \textit{strictly superharmonic} when we have in addition that $-\Delta u\not\equiv 0$.\\
\noindent Let $N\geq 2$. $\Omega\subset\R^N$ is a \textit{domain} when it is open and connected; moreover, $\Omega$ has \textit{a boundary of class} $C^{k,1}$ if $\dOmega$ can be described in local coordinates by a $C^{k}$ function with Lipschitz continuous $k$-th derivatives. Finally, $\Omega$ satisfies a \textit{uniform external ball condition} if there exists $R>0$ such that $\forall x\in\dOmega$ there exists a ball $B_R$ of radius $R$ such that $x\in\dB_R$ and $B_R\subset\R^N\setminus\Omegabar$.\\
\noindent The topological dual of a normed space $X$ is denoted by $X^*$; for $q\in[1,+\infty]$, $\|\cdot\|_q$ stands for the $\Lq(\Omega)$ norm and
$$\||\nabla^k\cdot|\|_q:=\biggl(\sum_{|\alpha|=k}\|D^\alpha\cdot\|_q^q\biggr)^{\frac{1}{q}},$$
where $\alpha$ is a multi-index.\\
Let us also recall that, if $\Omega\subset\R^2$ is a bounded domain with Lipschitz boundary, by Sobolev embeddings (see \cite[Theorem 4.12, Part II]{A}), $H^2(\Omega)\hookrightarrow C^{0,\lambda}(\Omegabar)$ for any $\lambda\in(0,1)$, thus we have continuous embedding in $\Lq(\Omega)$ for every $q\in[1,\infty]$.\\
Finally, we present some very useful facts about equivalence of norms in $H^2(\Omega)\cap\Ho$. The quoted results have been already obtained by Nazarov, Stylianou and Sweers in \cite{NStyS}; we will include the proof of the second equivalence in order to have a self-contained exposition.

\begin{lem}\label{eqnorm}
	Let $\Omega\subset\R^N$ bounded with a Lipschitz boundary and $\sigma\in(-1,1]$.
	\begin{enumerate}[label=(\roman*)]
		\item $\||\nabla^2\cdot|\|_2$ and $\|\cdot\|_{H^2(\Omega)}$ are equivalent norms on $H^2(\Omega)\cap\Ho$.
		\item If $\sigma=1$, assume additionally that $\Omega$ satisfying a uniform external ball condition. Then
		\begin{equation}\label{usualequivnorm}
		\|u\|_{H_\sigma(\Omega)}:=\bigg(\int_\Omega(\Delta u)^2 - 2(1-\sigma)\int_\Omega det(\nabla^2u)\bigg)^{\frac{1}{2}}
		\end{equation}
		defines a norm on $H^2(\Omega)\cap\Ho$ equivalent to the standard norm.
	\end{enumerate}
\end{lem}
\begin{proof}
	We prove here only (ii) and we refer to \cite[Corollary 5.4]{NStyS} for a proof of (i).	Firstly
	\begin{equation*}
	\begin{split}
	\|u\|_{H_\sigma(\Omega)}^2&=\int_\Omega u_{xx}^2+u_{yy}^2+2u_{xy}^2+2\sigma(u_{xx}u_{yy}-u_{xy}^2)\\
	&\leq\||\nabla^2u|\|_2^2+2|\sigma|\bigg(\dfrac{u_{xx}^2+u_{yy}^2}{2}+u_{xy}^2\bigg)=(1+|\sigma|)\||\nabla^2u|\|_2^2.
	\end{split}
	\end{equation*}
	Moreover, if $\sigma\in(-1,1)$, one has
	\begin{equation}\label{ineqnorm}
	\begin{split}
	\|u\|_{H_\sigma(\Omega)}^2&=\int_\Omega u_{xx}^2+u_{yy}^2+2(1-\sigma)u_{xy}^2+2\sigma u_{xx}u_{yy}\\
	&\geq\int_\Omega u_{xx}^2+u_{yy}^2+2(1-\sigma)u_{xy}^2-|\sigma|(u_{xx}^2+u_{yy}^2)\geq(1-|\sigma|)\||\nabla^2u|\|_2^2.
	\end{split}
	\end{equation}
	The proof is completed applying (i) and noticing that the map
	\begin{equation*}
	(u,v)_{H_\sigma}\mapsto\int_\Omega\Delta u\Delta v -(1-\sigma)\int_\Omega u_{xx}v_{yy}+u_{yy}v_{xx}-2u_{xy}v_{xy}
	\end{equation*}
	defines a scalar product on $H^2(\Omega)\cap\Ho$ for every $\sigma\in(-1,1)$ by the inequality \eqref{ineqnorm}. In the special case $\sigma=1$, one has $\|u\|_{H_1(\Omega)}=\|\Delta u\|_2$, which is an equivalent norm on $H^2(\Omega)\cap\Ho$ provided the external ball condition is satisfied (see \cite{Adolfsson}).
\end{proof}
\noindent In the following, $C_0=C_0(\Omega)$ and $C_A=C_A(\Omega)$ indicate the smallest positive constants such that 
\begin{equation}\label{C_0}
\|u\|_{H^2(\Omega)}^2:=\|u\|_2^2+\||\nabla u|\|_2^2+\||\nabla^2u|\|_2^2\leq C_0\||\nabla^2u|\|_2^2
\end{equation}
and 
\begin{equation}\label{C_A}
\|u\|_{H^2(\Omega)}^2\leq C_A\|\Delta u\|_2^2
\end{equation}
for every $u\in H^2(\Omega)\cap\Ho$.

\section{Existence of ground states}

In this section we investigate the existence of critical points of the generalized Kirchhoff-Love functional $J_\sigma:H^2(\Omega)\cap\Ho\rightarrow\R$ defined in \eqref{J}
in the physical relevant interval $\sigma\in(-1,1]$. Hereafter, we assume $\Omega$ to be a bounded domain in $\R^2$. Concerning the nonlinearity, the functional $J_\sigma$ is well-defined once we impose $F(\cdot, s)\in L^1(\Omega)$ and $F(x,\cdot)\in C^1(\R)$ (and thus there exists $f(x,\cdot)$ continuous such that $F(x,s)=\int_0^sf(x,t)dt$) and a power-type growth control on $F$, namely the existence of $a,b\in L^1(\Omega)$ such that $|F(x,s)|\leq a(x)+b(x)|s|^q$ for some $q>0$. With these assumptions on $F$, it is a standard fact to prove that $J_\sigma$ is a $C^1$ functional with $\Frechet$ derivative
$$J_\sigma'(u)[v]=\int_\Omega\Delta u\Delta v-(1-\sigma)\int_\Omega (u_{xx}v_{yy}+u_{yy}v_{xx}-2u_{xy}v_{xy})-\int_\Omega f(x,u)v.$$

\noindent Notice that, if $\Omega$ satisfies the assumptions of Lemma \ref{eqnorm}, we can rewrite the functional as
\begin{equation*}\label{Jequiv}
J_\sigma(u)=\dfrac{1}{2}\|u\|_{H_\sigma(\Omega)}^2-\int_\Omega F(x,u).
\end{equation*}

\noindent Our aim is to investigate the \textit{ground states} of the functional $J_\sigma$, i.e. the critical points on which the functional assumes the lowest value. In fact, besides the interest from a physical point of view, we are able to characterize them variationally and thus to apply a larger number of analytical tools.\\
\noindent Since the geometry of the functional plays an important role, from now on we have to distinguish between the sublinear case, that is, when the density $f$ has at most a slow linear growth in the real variable (as it will be specified in the following), and the superlinear case, the opposite one. In fact, we will see that in the first case $J_\sigma$ behaves similarly to the linear Kirchhoff-Love functional studied in \cite{PS} since it is coercive and ground states are global minima, while, in the second case, $J_\sigma$ has a mountain pass geometry and the ground states are saddle points. Moreover, although in the sequel we will be mainly interested in the power-type nonlinearity as in \eqref{FNL}, in the sublinear case we can easily generalize our analysis to a larger class of nonlinearities, as specified in Proposition \ref{coercKLg}.\\
\noindent We exclude from our analysis the case of general linear growths for the nonlinearity, for instance $f(x,u)=\lambda g(x)u$, since \eqref{PDEpSteklovg} becomes an eigenvalue problem and can be investigated with standard techniques (see also \cite[Theorem 4]{BGM}).

\subsection{Sublinear case}

\begin{prop}\label{coercKLg}
	With the assumptions for $\sigma$ and $\Omega$ as in Lemma \ref{eqnorm}, let $p\in(0,2)$ and suppose
	\begin{equation}\tag{H}\label{growth}
	|F(x,s)|\leq d(x)+c(x)|s|^p+\dfrac{1}{2}(1-|\sigma|)C_0^{-1}s^2
	\end{equation}
	where $c,d\in L^1(\Omega)$. Then the functional $J_\sigma$ is weakly lower semi-continuous and coercive, hence there exists a global minimizer of $J_\sigma$ in $H^2(\Omega)\cap\Ho$.
\end{prop}
\begin{proof}
	Let $(u_k)_{k\in\N}\subset H^2(\Omega)\cap\Ho\ni u$ be such that $u_k\rightharpoonup u$ weakly in $H^2(\Omega)$; since it is bounded in $H^2(\Omega)$ and consequently in $L^\infty(\Omega)$, one has
	$$|F(x,u_k)|\leq d(x)+c(x)M^p+\dfrac{1}{2}(1-|\sigma|)C_0^{-1}M^2,$$
	for some $M>0$, which is integrable over $\Omega$. Moreover, by the compactness of the embedding $H^2(\Omega)\hookrightarrow L^p(\Omega)$, there exists a subsequence $(u_{k_j})_{j\in\N}$ such that $u_{k_j}\rightarrow u$ in $L^p(\Omega)$ for a suitable $p\geq 1$, so $F(x,u_{k_j}(x))\rightarrow F(x,u(x))$ a.e. in $\Omega$ by continuity of $F(x,\cdot)$. Hence, by Dominated Convergence Theorem, we have $\int_\Omega F(x,u_{k_j})\rightarrow\int_\Omega F(x,u)$. This, together with the weakly lower semicontinuity of the norm, implies the same property for $J_\sigma$. If $\sigma\in(-1,1)$, by \eqref{C_0}:
	\begin{equation*}
	\begin{split}
	J_\sigma(u)&\geq\dfrac{1}{2}(1-|\sigma|)\||\nabla^2u|\|_2^2-\|d\|_1-C^p\|c\|_1\|u\|_{H^2(\Omega)}^p- \dfrac{1}{2}(1-|\sigma|)C_0^{-1}\|u\|_2^2\\
	&\geq\dfrac{1}{2}(1-|\sigma|)C_0^{-1}\||\nabla^2u|\|_2^2
	-\|c\|_1C^pC_0^{\frac{p}{2}}\||\nabla^2 u|\|_2^p-\|d\|_1;
	\end{split}
	\end{equation*}
	by (i) of Lemma \ref{eqnorm}, we deduce that $J_\sigma(u)\rightarrow+\infty$ as $\|u\|_{H^2(\Omega)}\rightarrow+\infty$, since $p\in(0,2)$.
	Easier computations provide a similar estimate to conclude the proof also if $\sigma=1$.
\end{proof}

\begin{remark}[model case]\label{modelcase}
	As an application of Proposition \ref{coercKLg}, we may consider the following kind of sublinearity: $F(x,u)=g(x)|u|^{p+1}+d(x)u$
	where $p\in(0,1)$ and $d,g\in L^1(\Omega)$.	In this case the functional is coercive and verifies (H). Notice also that if $g=0$ we retrieve the linear Kirchhoff-Love functional considered in \cite{PS}.
\end{remark}

\subsection{Superlinear case}

This case is more involved and we have to restrict to the nonlinearity \eqref{FNL} with $p>1$:
\begin{equation}\label{Jsuper}
J_\sigma(u):=\int_\Omega\dfrac{(\Delta u)^2}{2}- (1-\sigma)\int_\Omega det(\nabla^2 u)-\int_\Omega\dfrac{g(x)|u|^{p+1}}{p+1}.
\end{equation}
Here the functional is not coercive anymore: in fact, fixing any $u\in H^2(\Omega)\cap\Ho\setminus\{0\}$, we have $J_\sigma(tu)\rightarrow-\infty$ as $t\rightarrow+\infty$.
Following some arguments of \cite{CCN,PG}, we will make use of the method of the Nehari manifold to infer the existence of a (nontrivial) critical point. After some preliminary results, we will show that in our manifold the infimum of $J_\sigma$ is attained and then, using a deformation lemma, we will prove it is a critical point for $J_\sigma$ in $H^2(\Omega)\cap\Ho$.
\vskip0.2truecm
Let us define the \textit{Nehari manifold} of $J_\sigma$ as the set
$$\Ne_\sigma:=\{u\in (H^2(\Omega)\cap\Ho)\setminus\{0\} \,|\, J_\sigma'(u)[u]=0\},$$
which clearly contains all nontrivial critical points of $J_\sigma$. First of all, notice that $u\in\Ne_\sigma$ if and only if
$$\int_\Omega(\Delta u)^2-2(1-\sigma)\int_\Omega det(\nabla^2u)=\int_\Omega g(x)|u|^{p+1},$$
so one has the following two equivalent formulations for $J_\sigma$ restricted on $\Ne_\sigma$:
\begin{equation}\label{equivonNehari}
{J_\sigma}_{|_{\Ne_\sigma}}(u)=\bigg(\dfrac{1}{2}-\dfrac{1}{p+1}\bigg)\int_\Omega g(x)|u|^{p+1}=\bigg(\dfrac{1}{2}-\dfrac{1}{p+1}\bigg)\bigg(\int_\Omega(\Delta u)^2-2(1-\sigma)\int_\Omega det(\nabla^2u)\bigg),
\end{equation}
which implies ${J_\sigma}_{|\Ne_\sigma}(u)>0$ for every $u\in\Ne_\sigma$.

\noindent A crucial step will be to study what happens on the half-lines of $H^2(\Omega)\cap\Ho$:

\begin{lem}\label{t*lemma}
	Let $u\in H^2(\Omega)\cap\Ho\setminus\{0\}$ and the half-line $r_u:=\{tu\,|\,t> 0\}$. The intersection between $r_u$ and $\Ne_\sigma$ consists in a unique point $t^*(u)u$, where
	\begin{equation}\label{t*}
	t^*(u):=\bigg(\dfrac{\int_\Omega(\Delta u)^2-2(1-\sigma)\int_\Omega det(\nabla^2 u)}{\int_\Omega g(x)|u|^{p+1}}\bigg)^{\frac{1}{p-1}}.
	\end{equation}
	Moreover $J_\sigma(t^*(u)u)=\displaystyle\max_{t> 0}J_\sigma(tu)$.
\end{lem}
\begin{proof}
	For $t>0$ and a fixed $u\in H^2(\Omega)\cap\Ho\setminus\{0\}$, $tu\in\Ne_\sigma$ if and only if
	\begin{equation*}
	t^2\biggl[\int_\Omega(\Delta u)^2-2(1-\sigma)\int_\Omega det(\nabla^2 u)\biggr]=t^{p+1}\int_\Omega g(x)|u|^{p+1},
	\end{equation*}
	from which we deduce $t=t^*(u)$.
	Moreover, define
	$$\eta(t):=J_\sigma(tu)=\dfrac{t^2}{2}\biggl[\int_\Omega(\Delta u)^2-2(1-\sigma)\int_\Omega det(\nabla^2 u)\biggr]-\dfrac{t^{p+1}}{p+1}\int_\Omega g(x)|u|^{p+1}.$$
	If we look for $\bar{t}>0$ such that $\eta'(\bar t)=0$, we find again that $\bar{t}=t^*(u)$ and, since $\eta'(t)(t-t^*(u))<0$ for $t\neq t^*(u)$, we have that $t^*(u)u$ is the unique global maximum in the half-line $r_u$.
\end{proof}

\begin{lem}\label{Neclosed}
	The Nehari manifold is bounded away from $0$, i.e. $0\notin\overline{\Ne_\sigma}$.
\end{lem}
\begin{proof}
	Suppose first $\sigma\in(-1,1)$ and let $u\in H^2(\Omega)\cap\Ho\setminus\{0\}$. By Lemma \ref{eqnorm}, Lemma \ref{t*lemma} and the embedding $H^2(\Omega)\hookrightarrow L^\infty(\Omega)$, the following chain of inequalities holds:
	\begin{equation*}
	\begin{split}
	(1+|\sigma|)\|t^*(u)u\|_{H^2(\Omega)}^2&\geq\|t^*(u)u\|_{H_\sigma(\Omega)}^2\\
	&=(t^*(u))^{p+1}\int_\Omega g(x)|u|^{p+1}\\
	&\geq(C_0^{-1}(1-|\sigma|))^{\frac{p+1}{p-1}}\dfrac{\|u\|_{H^2(\Omega)}^\frac{2(p+1)}{p-1}}{(\int_\Omega g(x)|u|^{p+1})^\frac{2}{p-1}}\\
	&\geq C(\Omega,p,\sigma)\dfrac{\|u\|_{H^2(\Omega)}^\frac{2(p+1)}{p-1}}{(\|g\|_1\|u\|_{H^2(\Omega)}^{p+1})^\frac{2}{p-1}}=\dfrac{C(\Omega,p,\sigma)}{\|g\|_1^{\frac{2}{p-1}}}.
	\end{split}
	\end{equation*}
	If $\sigma=1$, one can deduce the same result using the equivalent norm on $H^2(\Omega)\cap\Ho$ given by $\|\Delta\cdot\|_2$. In both cases, there exists a uniform bound from below for the $H^2(\Omega)$ norm of the elements in Nehari manifold and thus $0$ cannot be a cluster point for $\Ne_\sigma$.
\end{proof}

\begin{prop}\label{infattained}
	There exists $u\in\Ne_\sigma$ such that $J_\sigma(u)=\displaystyle\inf_{v\in\Ne_\sigma}J_\sigma(v)=:c$
\end{prop}
\begin{proof}
	As already noticed, $c\geq 0$, since it attains positive values on $\Ne_\sigma$.
	Let now $(u_k)_{k\in\N}\subset\Ne_\sigma$ be a minimizing sequence for $J_\sigma$: we claim that $(u_k)_{k\in\N}$ is bounded in $H^2(\Omega)$ norm. In fact, if $\sigma\in(-1,1)$, there exists a constant $C>0$ such that, for every $k\in\N$,
	\begin{equation*}
	C\geq J_\sigma(u_k)=\biggl(\dfrac{1}{2}-\dfrac{1}{p+1}\biggr)\|u_k\|_{H_\sigma(\Omega)}^2\geq \biggl(\dfrac{1}{2}-\dfrac{1}{p+1}\biggr)(1-|\sigma|)C_0^{-1}\|u_k\|_{H^2(\Omega)}^2,
	\end{equation*}
	while \eqref{C_A} provides the right estimate in case $\sigma=1$. Hence, there exists a subsequence $(u_{k_j})_{j\in\N}\subset\Ne_\sigma$ and $u\in H^2(\Omega)\cap\Ho\setminus\{0\}$ such that $u_{k_j}\rightharpoonup u$ weakly in $H^2(\Omega)$ (and so weakly in $(H^2(\Omega)\cap\Ho,\|\cdot\|_{H\sigma})$ by Lemma \ref{eqnorm}) and strongly in $L^\infty(\Omega)$, by compact embedding. Consider now $t^*=t^*(u)$ such that $t^*u\in\Ne_\sigma$: by weak semicontinuity of the norm
	\begin{equation}\label{infattaineddim}
		\begin{split}
			&c=\inf_{v\in\Ne_\sigma}J_\sigma(v)\leq J(t^*u)=(t^*)^2\biggl[\int_\Omega\dfrac{(\Delta u)^2}{2}-(1-\sigma)\int_\Omega det(\nabla^2u )\biggr]-(t^*)^{p+1}\int_\Omega\dfrac{g(x)|u|^{p+1}}{p+1}\\
			&\leq \liminf_{j\rightarrow+\infty}\biggl((t^*)^2\biggl[\int_\Omega\dfrac{(\Delta u_{k_j})^2}{2}-(1-\sigma)\int_\Omega det(\nabla^2u_{k_j})\biggr]-(t^*)^{p+1}\int_\Omega\dfrac{ g(x)|u_{k_j}|^{p+1}}{p+1}\biggr)\\
			&=\liminf_{j\rightarrow+\infty}J_\sigma(t^*u_{k_j})\leq\liminf_{j\rightarrow+\infty}J_\sigma(u_{k_j})=c
		\end{split}
	\end{equation}
	where the last inequality holds because the supremum of $J_\sigma$ in each half-line $\{tu_{k_j}\,|\,t>0\}$ is achieved exactly in $u_{k_j}$ by Lemma \ref{t*lemma}. Hence, the infimum of $J_\sigma$ on $\Ne_\sigma$ is attained on $t^*u$.
\end{proof}

	\noindent In the proof of Proposition \ref{infattained} something weird happened: we took a minimizing sequence, which converges to an element $u$ and we proved that there exists $\alpha=t^*(u)\in\R$ such that $\alpha u$ is the minimum point of our functional $J_\sigma$. One expects that the minimum is $u$ itself and not a dilation of it. Indeed, one may prove that $t^*=1$.
	In fact, with the same notation as in that proof, from \eqref{infattaineddim} we deduce $J_\sigma({u_k}_j)\rightarrow c=J_\sigma(t^*u)$ by construction and $t^*u\in\Ne_\sigma$, so
	$$J_\sigma({u_k}_j)\rightarrow\bigg(\dfrac{1}{2}-\dfrac{1}{p+1}\bigg)\int_\Omega g(x)|t^*u|^{p+1}.$$
	Moreover, we took the sequence to be in the Nehari manifold itself, so $J_\sigma({u_k}_j)=(\frac{1}{2}-\frac{1}{p+1})\int_\Omega g(x)|{u_k}_j|^{p+1}$,
	and we have that ${u_k}_j\rightarrow u$ strongly in $L^\infty(\Omega)$, thus
	$$J_\sigma({u_k}_j)\rightarrow\bigg(\dfrac{1}{2}-\dfrac{1}{p+1}\bigg)\int_\Omega g(x)|u|^{p+1}.$$
	By the uniqueness of the limit, we must have $t^*=1$, so $u\in\Ne_\sigma$.

\begin{thm}\label{criticalpoint}
	The minimum $u$ of $J_\sigma$ in $\Ne_\sigma$ is a critical point for $J_\sigma$ in $H^2(\Omega)\cap\Ho$.
\end{thm}
\begin{proof}
	Suppose by contradiction that $u$ is not a critical point. Since the functional is $C^1$, there exists a ball centered in $u$ and $\varepsilon>0$ such that, for all $v\in B$,
	$$c-\varepsilon\leq J_\sigma(v)\leq c+\varepsilon,$$
	$$\|J_\sigma'(v)\|_{(H^2(\Omega)\cap\Ho)^*}\geq\dfrac{\varepsilon}{2},$$
	where $c=J_\sigma(u)=\displaystyle\inf_{v\in\Ne_\sigma}J_\sigma(v)$.
	Notice that on the half-line $r_u$, the point $u$ is the global maximum, so $J_\sigma(v)<c$ for each $v\in B\cap r_u$, $v\neq u$.\\
	If we denote $a=c-\varepsilon$, $b=c+\varepsilon$, $\delta=8$, $S=\overline{B_r(u)}$ and $S_0=\overline{H^2(\Omega)\cap\Ho\setminus B'}$, where $r>0$ such that $B_r(u)\subset\subset B'\subset\subset B$, 
	applying \cite[Proposition 5.1.25]{GP}, there exists a locally Lipschitz homotopy of homeomorphisms $\Gamma_t$ on $H^2(\Omega)\cap\Ho$ such that:
	\begin{enumerate}[label=(\roman*)]
		\item $t\mapsto J_\sigma(\Gamma(t,v))$ is decreasing in $B_r(u)$ and, in general, non-increasing;
		\item $J_\sigma(\Gamma(t,v))=v$ for $v\in S_0$ and $t\in [0,1]$, and so also for all $v\in\partial B$.
	\end{enumerate}
	From (i) we deduce that $J_\sigma(\Gamma(t,v))<c$ for every $v\in B\cap r_u$ and $t\neq 0$.
	Moreover, define the following map: $\psi:B\cap r_u\rightarrow\R$ such that
	$$\psi(v):=J_\sigma'(\Gamma(1,v))[\Gamma(1,v)]$$
	and consider $v\in\partial B\cap r_u$, so there exists $\alpha\neq 1$ such that $v=\alpha u$: we know from (ii) that $\Gamma(1,v)=v$ and, by Lemma \ref{t*lemma}, $J_\sigma'(\alpha u)[\alpha u]>0$ if $\alpha\in(0,1)$ and $J_\sigma'(\alpha u)[\alpha u]<0$ if $\alpha\in(1,+\infty)$, so $\psi(v)(v-u)<0$ on $\partial B\cap r_u$. As a result, since one can think at $\psi$ as a continuous map from $[x_1,x_2]\rightarrow\R$, where $x_1$ and $x_2$ correspond to the intersections between the half line $r_u$ and the ball $B$, and since $\psi(x_1)>0$ and $\psi(x_2)<0$, there exists a zero of $\psi$ in $(x_1,x_2)$, i.e. there exists $\bar{v}\in B\cap r_u$ such that $J_\sigma'(\Gamma(1,\bar{v}))[\Gamma(1,\bar{v})]=0$.\\
	Setting $w:=\Gamma(1,\bar{v})$, we have $w\in\Ne_\sigma$ and $J_\sigma(w)=J_\sigma(\Gamma(1,\bar{v}))<c=\displaystyle\inf_{v\in\Ne_\sigma}J_\sigma(v)$, a contradiction.
\end{proof}

	So far, we proved existence of a ground state for $J_\sigma$. Actually, one can say more about existence of general critical points by means of the Krasnoselski genus theory (see \cite[Section 10.2]{AM}). In fact, since our framework is subcritical, it is quite standard to prove the Palais-Smale condition for $J_\sigma$ by compact embedding of $H^2(\Omega)$ in every Lebesgue space. Moreover, our functional is $C^1$, even and bounded from below on the unit sphere of $H^2(\Omega)\cap\Ho$: indeed, 
	if $\|u\|_{H_\sigma(\Omega)}=1$, then $\|u\|_\infty<C$ for some $C>0$, so
	\begin{equation*}
	J_\sigma(u)=\dfrac{1}{2}-\int_\Omega\dfrac{g(x)|u|^{p+1}}{p+1}\geq\dfrac{1}{2}-\dfrac{C^{p+1}\|g\|_1}{p+1}>-\infty.
	\end{equation*}
	Hence, by \cite[Proposition 10.8]{AM}, one can ensure the existence of an infinite number of couples of critical points. The same argument may also be applied for the general sublinear case, provided $F(x,s)=F(x,-s)$ for every $s\in\R$.

\section{An identity and the positivity of ground states in convex domains}

The aim of this section is to prove positivity for the ground states found in the previous section. Notice that the problematic term in $J_\sigma$ is the one involving the determinant of the Hessian matrix. In order to overcome this difficulty, we need to rewrite it in an equivalent way, transforming it into a boundary term which can be handled in order to prove the desired positivity. Nevertheless, since the signed curvature of the boundary will be involved, we need to impose some regularity on $\dOmega$. We will basically deduce the same statement as Lemma 2.5 ($ii$) of \cite{PS}, but extending it to a larger class of domains.

\subsection{A crucial identity}

Our goal is to generalize the following result by Parini and Stylianou:
\begin{thm}[\cite{PS}, Lemma 2.5]\label{PS}
	Let $\Omega$ be a bounded domain in $\R^2$ with $C^{2,1}$ boundary, and let $\kappa$ be its signed curvature. Then for all $u\in H^2(\Omega)$, for every $\varphi\in H^3(\Omega)$, defining $K(u):=\int_\Omega det(\nabla^2u)dx$, we have:
	\begin{equation}\tag{F$_{PS}$}
	<K'(u),\varphi>=\int_{\partial\Omega}(\kappa \varphi_nu_n+\varphi_{\tau\tau}u_n-\varphi_{\tau n}u_\tau).
	\end{equation}
	Hence, for all $u\in H^2(\Omega)\cap\Ho$,
	\begin{equation}\tag{F}
	K(u)=\dfrac{1}{2}\int_{\partial\Omega}\kappa u_n^2.
	\end{equation} 
\end{thm}
\noindent Going into the details of its proof, one can actually realize that the strong regularity assumption on the boundary was needed only to derive (F) from (F$_{PS}$) because the authors used the density of $H^3(\Omega)\cap\Ho$ in $H^2(\Omega)\cap\Ho$, which strongly relied on the fact that $\partial\Omega\in C^{2,1}$ (see \cite[Lemma 2.3]{PS}). Nevertheless, (F$_{PS}$) requires only that all the elements therein are well defined. Hence, our starting point is the following:
\begin{cor}\label{startingpoint}
	Let $\Omega\subset\R^2$ be a bounded domain of class $C^{1,1}$.
	Then for every $v\in C^{\infty}(\Omegabar)$:
	\begin{equation}\tag{F$_{PS}$2}
	K(v)=\dfrac{1}{2}<K'(v),v>=\dfrac{1}{2}\int_{\partial\Omega}(\kappa v_n^2-(v_{n\tau}+v_{\tau n})v_\tau)
	\end{equation}
\end{cor}
\begin{proof}
	One only has to notice that if $\dOmega\in C^{1,1}$, $\kappa$ is well-defined in $L^\infty(\dOmega)$ and $$\int_\dOmega(v_{n\tau}v_\tau+v_nv_{\tau\tau})=\int_\dOmega(v_nv_\tau)_\tau=0$$
	as $\dOmega$ is a closed curve and by the definition of the tangential derivative (i.e. as $\frac{d}{ds}u(\gamma(s))$, where $\gamma$ is the parametrization of the curve $\dOmega$ in the arch parameter $s$).
\end{proof}

\noindent Our strategy consists of two steps: using (F$_{PS}$2), we will firstly prove that (F) holds also for every $v\in C^{1,1}_0(\Omegabar):=\{u\in C^{1,1}(\Omegabar)\,\,|\,\,u_{|\dOmega}=0\}$;
then, by a density result, we will transfer (F) from $C^{1,1}_0(\Omegabar)$ to $H^2(\Omega)\cap\Ho$. We will make use of the following Lemma, which makes a well-known result more precise:

\begin{lem}\label{paynerevisited}
	Let $\Omega\subset\R^N$ be a bounded domain of class $C^1$ and $u\in C^{1,1}(\Omegabar)$. Then there exists a sequence $(u_k)_{k\in\N}\in C^\infty(\Omegabar)$ such that $u_k\rightarrow u$ in $H^2(\Omega)$ and $\|u_k\|_{W^{2,\infty}(\Omega)}\leq C\|u\|_{W^{2,\infty}(\Omega)}$ for some positive constant $C$.
\end{lem}
\begin{proof}
	First of all notice that $C^{1,1}(\Omegabar)$ can be equivalently seen as $W^{2,\infty}(\Omega)$, which is a subset of $H^2(\Omega)$ since $\Omega$ is a bounded domain; moreover the fact that $C^\infty(\Omegabar)$ is dense in $H^2(\Omega)$ in $H^2(\Omega)$ norm if $\dOmega$ is of class $C^1$ is a standard fact (see \cite[section 5.3.3, Theorem 3]{E}), so the only statement to be verified is the $W^{2,\infty}(\Omega)$ estimate. Since the main tool in the proof of the $H^2(\Omega)$ convergence is the local approximation, which is achieved by mollification, we only have to prove that the same inequality holds there. So, let $v\in L^\infty(\Omega)$, $\varepsilon>0$ and consider
	$$v_\varepsilon(x):=(\eta_\varepsilon\ast v)(x)=\int_{B_\varepsilon(0)}\eta_\varepsilon(y)v(x-y)dy,$$
	where $\eta_\varepsilon$ is the standard mollifier in $\R^N$, that is $\eta_\varepsilon:=\varepsilon^{-n}\eta(\frac{x}{\varepsilon})$ and
	$$\eta(x)=\tilde{C}e^{\frac{1}{|x|^2-1}}\chi_{B_1(0)}(x),$$
	where $\tilde{C}>0$ such that $\int_{B_1(0)}\eta(z)dz=1.$
	So $v_\varepsilon$ is well-defined in $\Omega_\varepsilon:=\{x\in\Omega\,\,|\,\,d(x,\dOmega)>\varepsilon\}$ and we know that $v_\varepsilon\in C^\infty(\Omega_\varepsilon)$ and $\eta_\varepsilon$ is such that $\int_{B_\varepsilon(0)}\eta_\varepsilon(z)dz=1$.\\
	\noindent We claim that $\|v_\varepsilon\|_{L^\infty(\Omega_\varepsilon)}\leq \|v\|_{L^\infty(\Omega)}$. In fact,
	\begin{equation*}
	\|v_\varepsilon\|_{L^\infty(\Omega_\varepsilon)}\leq\sup_{x\in\Omega_\varepsilon}\int_{B_\varepsilon(0)}|\eta_\varepsilon(z)||v(x-z)|dz\leq\|v\|_{L^\infty(\Omega)}\int_{B_\varepsilon(0)}|\eta_\varepsilon(z)|dz=\|v\|_{L^\infty(\Omega)}.
	\end{equation*}
	Also for derivatives of $v$ the same inequality holds, because for any admissible multiindex $\alpha$ we know that $D^\alpha(v_\varepsilon)=(D^\alpha(v))_\varepsilon$ (see \cite[Lemma 7.3]{GT}).\\
	At this point, following the aforementioned proof of \cite{E}, it is easy to derive the desired result.	
\end{proof}

\begin{prop}\label{detbordo1}
	Let $\Omega\subset\R^2$ be a bounded domain of class $C^{1,1}$. Then, for all $u\in C^{1,1}_0(\Omegabar)$:
	$$\int_\Omega det(\nabla^2u)=\dfrac{1}{2}\int_\dOmega \kappa u_n^2.$$
\end{prop}
\begin{proof}
	Applying Lemma \ref{paynerevisited}, let $(u_k)_{k\in\N}\subset C^{\infty}(\Omegabar)$ be a sequence converging to $u$ in $H^2(\Omega)$, whose norms in $W^{2,\infty}$ are controlled by the $W^{2,\infty}$ norm of $u$. By Corollary \ref{startingpoint}, the following identity holds:
	\begin{equation}\label{FCk}
	K(u_k)=\dfrac{1}{2}\int_{\partial\Omega}[\kappa(u_k)_n^2-((u_k)_{n\tau}+(u_k)_{\tau n})(u_k)_\tau].
	\end{equation}
	By the convergence in $H^2(\Omega)$ one clearly has $K(u_k)\rightarrow K(u)$; moreover, since $\kappa\in L^\infty(\dOmega)$ and using Trace Theorem, one can deduce also that 
	$$\int_{\partial\Omega}\kappa(u_k)_n^2\rightarrow\int_\dOmega \kappa u_n^2.$$
	Finally we have to consider the terms in which tangential derivatives are involved. Similarly to the normal derivative, one has $(u_k)_\tau\rightarrow u_\tau$ in $L^2(\dOmega)$, so  $(u_k)_\tau\rightarrow 0$ in $L^2(\dOmega)$, since $u_{|\dOmega}=0$. Furthermore,
	$$(u_k)_{n\tau}=\nabla(u_k)_n\cdot\tau=\nabla(\nabla u_k\cdot n)\cdot\tau=(\nabla^2u_k\cdot n+\nabla u_k\cdot\nabla n)\cdot\tau$$
	and (see \cite[Chapter 4]{Sperb})
	$$(u_k)_{\tau n}=\sum_{i,j=1}^2\dfrac{\partial^2 u_k}{\partial x_i\partial x_j}\tau_i n_j$$
	and one can infer that $(u_k)_{n\tau}$ and $(u_k)_{\tau n}$ are uniformly bounded in $L^2(\dOmega)$. In fact, since $u_k$ are $C^\infty$ functions and using Lemma \ref{paynerevisited}:
	\begin{equation*}
	\begin{split}
	\|(u_k)_{n\tau}\|_{L^2(\dOmega)}&\leq|\dOmega|^{1/2}\|(u_k)_{n\tau}\|_{L^\infty(\dOmega)}\leq|\dOmega|^{1/2}(\||\nabla^2u_k\cdot n|\|_{L^\infty(\dOmega)}+\||\nabla u_k\cdot\nabla n|\|_{L^\infty(\dOmega)})\\
	&\leq2|\dOmega|^{1/2}\|n\|_{W^{1,\infty}(\dOmega)}\|u_k\|_{W^{2,\infty}(\Omega)}\leq C(\Omega)\|u\|_{W^{2,\infty}(\Omega)}
	\end{split}
	\end{equation*}
	and similarly for $(u_k)_{\tau n}$. Consequently,
	\begin{equation*}
	\int_{\dOmega}\bigg((u_k)_{n\tau}+(u_k)_{\tau n}\bigg)(u_k)_\tau\rightarrow 0.
	\end{equation*}
\end{proof}

\noindent In order to extend (F) to the space $H^2(\Omega)\cap\Ho$, we need a density result (Lemma \ref{StylianouRevisited} below) which is taken from \cite[Theorem 2.2.4]{Sty} and that can be adapted to our context: in fact, it concerns $C^2$ functions and diffeomorphisms but, with a little care, one can obtain the same result also in the class $C^{1,1}$.
\begin{defn}(\cite{A}, 3.40, p.77)
	Let $\varPhi$ be a one-to-one transformation of a domain $\Omega\subset\R^N$ onto a domain $G\subset\R^N$ having inverse $\varPsi:=\varPhi^{-1}$. We say that $\varPhi$ is a $C^{1,1}$ diffeomorphism if, writing $\varPhi=(\varPhi_1,...,\varPhi_N)$ and $\varPsi=(\varPsi_1,...,\varPsi_N)$, then $\varPhi_i\in C^{1,1}(\Omegabar)$ and $\varPsi_i\in C^{1,1}(\overline{G})$ for every $i\in\{1,...,N\}$.
\end{defn}

\begin{lem}\label{StylianouRevisited}
	Let $\Omega\subset\R^N$ be bounded and open such that for every $x\in\dOmega$ there exists a $j\in\{0,...,N-1\}$, $\varepsilon>0$ and a $C^{1,1}$-diffeomorphism $\varPhi:\R^N\rightarrow\R^N$, such that the following hold:
	\begin{itemize}
		\item $\varPhi(x)=0$;
		\item $\varPhi(B_\varepsilon(x)\cap\Omega)\subset S_j:=\{x=(x_1,...,x_N)\in\Omega\,|\,x_i>0\,,\forall i>j\}$;
		\item $\varPhi(B_\varepsilon(x)\cap\dOmega)\subset\partial S_j$.
	\end{itemize}
	Then:
	$$\overline{C_0^{1,1}(\Omegabar)}^{\|\cdot\|_{H^2(\Omega)}}=H^2(\Omega)\cap\Ho.$$
\end{lem}

\begin{thm}\label{detbordo2}
	Let $\Omega\subset\R^2$ be a bounded domain of class $C^{1,1}$. Then, for all $u\in H^2(\Omega)\cap\Ho$:
	\begin{equation}\tag{F}
	\int_\Omega det(\nabla^2u)=\dfrac{1}{2}\int_\dOmega \kappa u_n^2.
	\end{equation}
\end{thm}
\begin{proof}
Let $u\in H^2(\Omega)\cap\Ho$; since the assumptions on the boundary are clearly fulfilled if $\dOmega$ is of class $C^{1,1}$, applying Lemma \ref{StylianouRevisited} we get an approximating sequence $(u_k)_{k\in\N}\subset C^{1,1}_0(\Omegabar)$ converging in $H^2(\Omega)$ to $u$. With the same steps as in the proof of Proposition \ref{detbordo1}, by the $H^2(\Omega)$ convergence, we have both $K(u_k)\rightarrow K(u)$ and $\int_{\partial\Omega}\kappa(u_k)_n^2\rightarrow\int_\dOmega \kappa u_n^2$ and one concludes by the uniqueness of the limit.
\end{proof}

\subsection{From the functional to the PDE}
As already briefly mentioned in the introduction, if the boundary is smooth enough ($\dOmega$ of class $C^{4,\alpha}$ for $\alpha>0$), standard elliptic regularity results apply and one can integrate by parts the Euler-Lagrange equation from $J_\sigma$ to see that critical points satisfy
\eqref{PDEpSteklovg}.
On the other hand, assuming only that the boundary is of class $C^{1,1}$, the signed curvature is well-defined in $L^\infty(\Omega)$ and we can have a weak formulation of problem \eqref{PDEpSteklovg}. More precisely, in this case, by \textit{weak solution} of \eqref{PDEpSteklovg} here we mean a function $u\in H^2(\Omega)\cap\Ho$ which satisfies
\begin{equation}\label{defsol}
\int_\Omega\Delta u\Delta\varphi-(1-\sigma)\int_\dOmega \kappa u_n\varphi_n = \int_\Omega g(x)|u|^{p-1}u\varphi\quad\quad\forall\varphi\in H^2(\Omega)\cap\Ho.
\end{equation}
Consequently, we can equivalently say "ground states of $J_\sigma$" or "ground state solutions for \eqref{PDEpSteklovg}". For a proof of the equivalence of the two problems, we refer to \cite{GS}.

\subsection{Positivity of ground states in convex domains}
Assuming that $\dOmega$ is of class $C^{1,1}$, Theorem \ref{detbordo2} enables us to rewrite the functional $J_\sigma$ in a more convenient way: in fact, we deduce that for every $u\in H^2(\Omega)\cap\Ho$, 
\begin{equation}\label{Jnew}
J_\sigma(u)=\int_\Omega\dfrac{(\Delta u)^2}{2} - \dfrac{1-\sigma}{2}\int_\dOmega\kappa u_n^2-\int_\Omega F(x,u),
\end{equation}
where we recall that $F(x,s)=\int_0^sf(x,t)dt$.\\
With this formulation, now we are able to establish the positivity of ground states of the functional $J_\sigma$ in convex domains with boundary of class $C^{1,1}$ if the density function $f(x,u)$ is nonnegative, both in sublinear and superlinear case. We will make use of the method of the superharmonic function, which is quite a standard tool when dealing with fourth order problems and which has already been successfully used, for instance, in \cite{BGM}, \cite{GS} or \cite{NStyS} and whose core is contained in the following lemma:

\begin{lem}\label{StandardArgumentPPP}
	Let $\Omega\subset\R^N$ be a bounded convex domain; fix $u\in H^2(\Omega)\cap\Ho$ and define $\tilde u$ as the unique solution in $\Ho$ of the following Poisson problem:
	\begin{equation}\label{tildePb}
	\begin{cases}
	-\Delta\tilde{u}=|\Delta u|\quad&\mbox{in }\Omega\\
	\tilde{u}=0\quad&\mbox{on }\dOmega.
	\end{cases}
	\end{equation}
	Then $\tilde u\in H^2(\Omega)\cap\Ho$ and either $\tilde u>|u|$ in $\Omega$ and ${\tilde u}_n^2\geq u_n^2$ on $\dOmega$ or $\tilde u=u$ in $\Omega$.
\end{lem}

\begin{proof}
	Since $\Omega$ is convex by assumption, it satisfies in particular a uniform external ball condition and thus, by \cite{Adolfsson}, we infer $\tilde u\in H^2(\Omega)$. Suppose $\tilde u\not\equiv u$. Since in particular $-\Delta \tilde u\geq\Delta u$ holds, by the maximum principle for strong solutions (see \cite[Theorem 9.6]{GT}), one has $\tilde u> -u$ in $\Omega$ and so $\tilde{u}_n \leq -u_n$. Similarly, $-\Delta \tilde{u}\geq-\Delta u$, implies also $\tilde u> u$ and $\tilde{u}_n \leq u_n$ and so, combining them, we have the result.
\end{proof}

\begin{prop}\label{positivitysublinear}(Sublinear Case)
	Let $\Omega\subset\R^2$ be a bounded convex domain with $\dOmega$ of class $C^{1,1}$ and $\sigma\in(-1,1]$. In addition to the assumption (H), suppose also that $f\geq 0$ and positive for a subset of positive measure. If $u\in H^2(\Omega)\cap\Ho$ is a nontrivial minimizer of $J_\sigma$, then $u$ is strictly superharmonic in $\Omega$, thus positive.
\end{prop}	
\begin{proof}
	Firstly notice that $\kappa\geq 0$ a.e. on $\dOmega$ by the convexity of $\Omega$. From $u$, define its superharmonic function $\tilde u$ as in Lemma \ref{StandardArgumentPPP}. Supposing $\tilde u\not\equiv u$, by that result we can infer
	\begin{equation}\label{positivitysublinearproof}
	\begin{split}
	J_\sigma(\tilde{u})&=\int_\Omega\dfrac{(\Delta \tilde{u})^2}{2} - \dfrac{1-\sigma}{2}\int_\dOmega\kappa \tilde{u}_n^2-\int_\Omega F(x,\tilde{u})\\
	&\leq\int_\Omega\dfrac{(\Delta u)^2}{2}- \dfrac{1-\sigma}{2}\int_\dOmega\kappa u_n^2-\int_\Omega F(x,\tilde{u}).
	\end{split}
	\end{equation}
	Nevertheless, since $\frac{\partial F}{\partial s}=f\geq 0$, we have also that $F(x,u)<F(x,\tilde{u})$, and thus $J_\sigma(\tilde{u})< J_\sigma(u)$, which leads to a contradiction. Hence necessarily $\tilde u$ coincides with $u$, so $-\Delta u=-\Delta\tilde u=|\Delta u|\geq 0$. As $u=0$ on $\dOmega$ and $u\not\equiv0$, we deduce $u>0$ in $\Omega$.
\end{proof}

	It is clear that, when $f(x,0)\not=0$, by Proposition \ref{coercKLg}, we always find a \textit{nontrivial} global minimizer, which is positive by Proposition \ref{positivitysublinear}. For homogeneous nonlinearities this is not true in general. Anyway, for our model $f(x,s)=g(x)|s|^{p-1}s$, if we restrict our attention to the Nehari set, we easily see $J_\sigma(u)=(\frac{1}{2}-\frac{1}{p+1})\|u\|_{H_\sigma}<0$ for every $u\not=0$. So it is clear that in the minimization process we do not fall on 0. The same argument holds for more general nonlinearities $f(x,u)$, provided
	$$\dfrac{f(x,u)u}{2}-F(x,u)<0\qquad\mbox{for all }u\in H^2(\Omega)\cap\Ho.$$
	For instance this holds when $f(x,s)=g(x)|s|^{p-1}s+h(x)|s|^{q-1}s$, for $g,h>0$, $p,q\in(0,1)$.

\begin{remark}
	We stress here that, as a direct consequence of Proposition \ref{positivitysublinear}, we have obtained the positivity preserving property also in the case of $f$ not depending on $u$, i.e. for the linear Kirchhoff-Love functional $I_\sigma$ (cf. also Remark \ref{modelcase}). This generalizes the corresponding result by Parini and Stylianou \cite[Theorem 3.1]{PS} for bounded convex domains assuming only $C^{1,1}$ regularity on the boundary.
\end{remark}

In our sublinear model case $f(x,s)=g(x)|s|^{p-1}s$, $p\in(0,1)$, something more may be deduced: in fact, Lemma \ref{t*lemma} still applies and, with the same steps as in the proof of Lemma \ref{Neclosed}, (reversing the inequalities since now $p-1<0$), one ends up with
$$\|u\|_{H^2(\Omega)}\leq\bigg(\dfrac{\|g\|_1C(\Omega)}{(1-|\sigma|)C_0^{-1}}\bigg)^{\frac{1}{1-p}}\quad\quad\mbox{for all}\,\,u\in\Ne_\sigma.$$
As a result, we can state the following:
\begin{prop}\label{upperboundH^2}
	Let $\Omega$ be a bounded Lipschitz domain in $\R^2$ and let $g\in L^1(\Omega)$ be positive a.e. in $\Omega$. For every $\sigma\in(-1,1)$ fixed, all critical points of $J_\sigma$ with $f(x,s)=g(x)|s|^{p-1}s$ and $p\in(0,1)$ are uniformly bounded in $H^2(\Omega)$.
\end{prop}
\noindent Notice that by continuous embedding $H^2(\Omega)\hookrightarrow L^\infty(\Omega)$, one may also infer an a-priori $L^\infty$ bound for all critical points of $J_\sigma$. The estimate becomes also uniform with respect to $\sigma$ if we restrict $\sigma\in I\Subset(-1,1)$.
\vskip0.2truecm
Concerning the superlinear case with the nonlinearity \eqref{FNL}, we obtain the same positivity result with the same assumptions on $\Omega$ and $\sigma$:

\begin{prop}\label{positivitysuperlinear}(Superlinear Case)
	Let $\Omega\subset\R^2$ be a bounded convex domain with $\dOmega$ of class $C^{1,1}$ and $\sigma\in(-1,1]$. Moreover suppose $f(x,u)=g(x)|u|^{p-1}u$, where $p>1$ and $g\in L^1(\Omega)$ positive a.e. in $\Omega$. Then the ground states of the functional $J_\sigma$ are positive in $\Omega$.
\end{prop}
\begin{proof}
	Suppose, by contradiction, that there exists $u\in\Ne_\sigma$ such that $J_\sigma(u)=\inf\{J_\sigma(v)\,|\,v\in\Ne_\sigma\}$ and $u$ is not positive. With the same spirit of the proof of Proposition \ref{positivitysublinear}, consider the superharmonic function $\tilde u$ associated to $u$ and suppose they are not the same. This time the inequality \eqref{positivitysublinearproof} is not sufficient to have a contradiction since we do not know whether $\tilde u\in\Ne_\sigma$. Nevertheless, by Lemma \ref{t*lemma}, there exists $t^*:=t^*(\tilde u)\in\R^+$ such that $t^*\tilde u\in\Ne_\sigma$. Then,
	\begin{equation}\label{poseq2}
	\begin{split}
	J_\sigma(t^*\tilde{u})&=(t^*)^2\biggl[\int_\Omega\dfrac{(\Delta\tilde{u})^2}{2}-\dfrac{1-\sigma}{2}\int_\dOmega \kappa \tilde{u}_n^2\biggr]-(t^*)^{p+1}\int_\Omega\dfrac{g(x)|\tilde{u}|^{p+1}}{p+1}\\
	&<(t^*)^2\biggl[\int_\Omega\dfrac{(\Delta u)^2}{2}-\dfrac{1-\sigma}{2}\int_\dOmega \kappa u_n^2\biggr]-(t^*)^{p+1}\int_\Omega\dfrac{g(x)|u|^{p+1}}{p+1}\\
	&=J_\sigma(t^*u)\leq J_\sigma(u),
	\end{split}
	\end{equation}
	 which is again a contradiction. Notice that the last inequality holds since, by Lemma \ref{t*lemma}, $J_\sigma$ restricted to every half-line attains its maximum on the Nehari manifold. Thus necessarily $\tilde u$ coincides with $u$, which implies that $u$ is strictly superharmonic and thus positive.
\end{proof}

\begin{remark}\label{extensionnotobvious}
	Notice that in both proofs of Propositions \ref{positivitysublinear} and \ref{positivitysuperlinear}, if $\sigma$ lies in the interval $(-1,1]$, the assumption $\Omega$ convex was necessary to have the good inequality for the second term of $J_\sigma$; on the other side, if $\sigma>1$ we do not have anymore the right sign and we cannot conclude the argument.
\end{remark}

\section{Beyond the physical bounds: $\sigma\leq-1$}
So far, we studied the existence of critical points of the functional $J_\sigma$ with the assumption $\sigma\in(-1,1]$, we described in a variational way the geometry of the ground states and we finally established their positivity. The aim of this section is to study what happens to the ground states of $J_\sigma$ if we let the parameter to be in the whole $\R$. Again, we are especially interested in studying their positivity.\\
Since the study is rather different if $\sigma\leq-1$ or $\sigma>1$, we divide the subject into two sections. In both, we will always assume that $\Omega\subset\R^2$ is a bounded convex domain of class $C^{1,1}$ so that Theorem \ref{detbordo2} holds. Moreover, as it seems  more interesting from a mathematical point of view, we mainly focus on the superlinear case $f(x,u)=g(x)|u|^{p-1}u$ with $p>1$, pointing out, if needed, the necessary adaptation for the sublinear power $p\in(0,1)$.

\subsection{A Steklov eigenvalue problem}
Let us begin by recalling some known facts about the eigenvalue problem associated to equation \eqref{PDEpSteklovg} (see \cite{GS} or, for the case $\kappa=1$, \cite{BucurFG} or \cite{BGM}):
\begin{equation}\label{PDEeig}
\begin{cases}
\Delta^2u=0\quad&\mbox{in }\Omega,\\
u=0\quad&\mbox{on }\dOmega,\\
\Delta u=d\kappa u_n\quad&\mbox{on }\dOmega.
\end{cases}
\end{equation}
We call \textit{Steklov eigenvalue} each real value $d$ such that \eqref{PDEpSteklovg} admits a nontrivial weak solution, named \textit{Steklov eigenfunction}, i.e. $u\in H^2(\Omega)\cap\Ho$, $u\neq 0$, such that for all $\varphi\in H^2(\Omega)\cap\Ho$
\begin{equation}\label{defeig}
\int_\Omega\Delta u\Delta\varphi-d\int_\dOmega \kappa u_n\varphi_n = 0.
\end{equation}
First of all, $d$ must be positive. In fact, if $u$ is a Steklov eigenfunction, taking $u=\varphi$ in \eqref{defeig}:
$$d\int_\dOmega\kappa(u_n)^2=\int_\Omega(\Delta u)^2>0,$$
since $\|\Delta\cdot\|_2$ is a norm in $H^2(\Omega)\cap\Ho$. As $\kappa\geq0$, we have both $d>0$ and $\int_\dOmega\kappa u_n^2>0$. As a complementary result, in order to show nontrivial solutions of \eqref{PDEeig}, without loss of generality, we can restrict to the subset
$$\mathcal{H}=\bigg\{u\in H^2(\Omega)\cap\Ho\,\bigg|\,\int_\dOmega\kappa (u_n)^2\neq 0\bigg\}.$$

\begin{defn}\label{deftildedelta1}
	We denote by $\tilde{\delta}_1(\Omega)$ the first Steklov eigenvalue for problem \eqref{PDEeig}:
	$$\tilde{\delta}_1(\Omega):=\inf_{\mathcal{H}\setminus\{0\}}\dfrac{\|\Delta u\|_2^2}{\int_\dOmega\kappa u_n^2}.$$
\end{defn}
\begin{prop}\label{eigenfunction}
	The first Steklov eigenvalue is attained, positive and there exists a unique (up to a multiplicative constant) corresponding Steklov eigenfunction, which is positive in $\Omega$.
\end{prop}
\begin{proof}
	We refer to \cite[Lemma 4.4]{GS}, just noticing that the continuity of the curvature assumed therein was not necessary to obtain this result.
\end{proof}

\subsection{A nonexistence and an existence result}
From Proposition \ref{eigenfunction}, it is easy to deduce a nonexistence result for positive solution if $\sigma$ is negative enough:
\begin{prop}\label{nonexistence}
	If $\sigma\leq\sigma^*:=1-\tilde{\delta}_1(\Omega)$, there is no nonnegative nontrivial solution for the Steklov boundary problem \eqref{PDEpSteklovg}.
\end{prop}
\begin{proof}
	Let $u$ be a nonnegative solution for \eqref{PDEpSteklovg} and $\Phi_1>0$ be the first Steklov eigenfunction; we use $\Phi_1$ as a test function in \eqref{defsol}:
	$$\int_\Omega\Delta u\Delta\Phi_1-(1-\sigma)\int_\dOmega \kappa u_n(\Phi_1)_n=\int_\Omega g(x)u^p\Phi_1$$
	and then, interpreting $u$ this time as a test function in \eqref{defeig}, we have
	$$\int_\Omega\Delta u\Delta\Phi_1=\tilde{\delta}_1(\Omega)\int_\dOmega \kappa(\Phi_1)_n u_n.$$
	Combining the two equalities,
	$$(\tilde{\delta}_1(\Omega)-(1-\sigma))\int_\dOmega\kappa(\Phi_1)_n u_n=\int_\Omega g(x)u^p\Phi_1>0.$$
	Again by positivity of $u$ and $\Phi_1$, we have $u_n\leq 0$ and $(\Phi_1)_n\leq 0$ so, as $\kappa\geq 0$, we finally end up with $\tilde{\delta}_1(\Omega)-1+\sigma>0$, which is exactly what we wanted.
\end{proof}

\begin{remark}
	We already proved that our problem \eqref{PDEpSteklovg} admits positive solutions whenever $\sigma\in(-1,1]$ with the same assumptions on $\Omega$. Hence, we infer that, $\tilde\delta_1(\Omega)\geq2$ and we have equality 
	if $\Omega=B_1(0)$ (see \cite[Proposition 12]{BGM}). This result was already proved for $C^2$ bounded convex domains of $\R^2$ by Parini and Stylianou in \cite[Remark 3.3]{PS}, using Fichera's duality principle.
\end{remark}

The next step is to investigate what happens if $\sigma\in(\sigma^*,-1]$ in case this interval is nonempty. We will show that the existence and the positivity results found for $\sigma\in(-1,1]$ can be extended for this case. In fact, the only restriction we have to overcome, is the fact that here Lemma \ref{eqnorm} is not the right way to prove that the first two terms in the functional $J_\sigma$ define indeed a norm on $H^2(\Omega)\cap\Ho$.
\begin{lem}\label{eqnormnew}
	For every $\sigma>\sigma^*$, the map
	$$u\mapsto\bigg[\int_\Omega(\Delta u)^2-(1-\sigma)\int_\dOmega \kappa(u_n)^2\bigg]^\frac{1}{2}:=\|u\|_{H_\sigma}$$
	is a norm in $H^2(\Omega)\cap\Ho$ equivalent to the standard norm.
\end{lem}
\begin{proof}
	By definition of $\tilde\delta_1(\Omega)$ as an inf, we have $\|\Delta u\|_2^2\geq\tilde\delta_1(\Omega)\int_\dOmega \kappa u_n^2$ for each $u\in H^2(\Omega)\cap\Ho$ and so, if $d>0$ (which corresponds to $\sigma<1$),
	\begin{equation}\label{ineqnorm2}
	\int_\Omega(\Delta u)^2\geq\int_\Omega(\Delta u)^2-d\int_\dOmega \kappa u_n^2\geq\int_\Omega(\Delta u)^2\bigg(1-\dfrac{d}{\tilde\delta_1(\Omega)}\bigg).
	\end{equation}
	On the other hand, if $d<0$ (so that $\sigma>1$),
	$$\int_\Omega(\Delta u)^2\leq\int_\Omega(\Delta u)^2+|d|\int_\dOmega\kappa u_n^2\leq\int_\Omega(\Delta u)^2\bigg(1+\dfrac{|d|}{\tilde\delta_1(\Omega)}\bigg).$$
	As a result, we have to impose that $d<\tilde\delta_1(\Omega)$ to have the positivity of the constant in the first estimate, while no restriction occurs in the second. The proof is completed noticing that the map
	$$(u,v)_{H_\sigma}\mapsto\int_\Omega\Delta u\Delta v-d\int_\dOmega\kappa u_nv_n$$
	defines a scalar product on $H^2(\Omega)\cap\Ho$ by inequality \eqref{ineqnorm2} for all $d<\tilde\delta_1(\Omega)$.
	
\end{proof}

\begin{prop}\label{extension1}
	Let $\Omega\subset\R^2$ be a bounded convex domain with boundary $C^{1,1}$ and suppose $\sigma\in(\sigma^*,-1]$; then the functional $J_\sigma$ admits a positive ground state.
\end{prop}
\begin{proof}
	It is sufficient to notice that Lemma \ref{Neclosed} holds for these values of $\sigma$ if we replace Proposition \ref{eqnorm} by Lemma \ref{eqnormnew}, while all the other propositions that led to the existence and the positivity of ground states are not affected by this change.
\end{proof}

\begin{remark}(Sublinear Case) Both Propositions \ref{nonexistence} and \ref{extension1} hold in the case of a function $f(x,u)$ which verifies the assumption (H) (modifying in a suitable way the constant in front of the quadratic term) and $f\geq 0$, $f\not\equiv 0$.	
\end{remark}

\subsection{Approaching $\sigma^*$}
As we know now the existence of positive ground state solutions for $\sigma\in(\sigma^*,1]$ and that there are no positive solutions if $\sigma\leq\sigma^*$, a natural question that arises is what is the behaviour of a sequence $(u_k)_{k\in\N}$, each of them being a ground state for the respective functional $J_{\sigma_k}$, as $\sigma_k\searrow\sigma^*$. We will find an antipodal result for $f(x,u)=g(x)|u|^{p-1}u$ as $p\in(1,+\infty)$ or $p\in(0,1)$.\\
\noindent The following proof is an adaptation of \cite[Theorem 1]{BGW}, which covers the critical case $f(x,u)=|u|^{2^*-2}u$, when the dimension $N\geq 5$. Moreover, the authors considered a slightly different notion of solution, that is, the minimizers of the Rayleigh quotient associated to the boundary value problem:
\begin{equation}\label{Rayleigh}
R_\sigma (u):=\dfrac{\|\Delta u\|_2^2-(1-\sigma)\int_\dOmega \kappa u_n^2}{\bigg(\int_\Omega g(x)|u|^{p+1}\bigg)^\frac{2}{p+1}}
\end{equation}
Anyway, it is a standard fact to prove that every ground state of $J_\sigma$ is also a minimizer of $R_\sigma$, while the converse is also true, up to a multiplication by a constant.
\begin{thm}\label{approachingsigma*}
	Let $\Omega$ as in Proposition \ref{extension1} and $\sigma_k\searrow\sigma^*$ as $k\rightarrow+\infty$. If $p\in(0,1)$, then $\|u_k\|_\infty\rightarrow+\infty$, while, if $p>1$, then $\|u_k\|_{H^2(\Omega)}\rightarrow 0$.
\end{thm}
\begin{proof}
	Let $p>0$, $p\neq 1$; by the remark above, each ground state $u_k$ is such that
	$$R_{\sigma_k}(u_k)=\inf_{0\neq u\in H^2(\Omega)\cap\Ho}R_{\sigma_k}(u):=\Sigma_{\sigma_k}\geq 0.$$
	By Proposition \ref{eigenfunction}, there exists a positive Steklov first eigenfunction $\Phi_1$; since we have $\|\Delta\Phi_1\|_2^2=(1-\sigma^*)\int_\dOmega \kappa(\Phi_1)_n^2$, then
	$$0\leq\Sigma_{\sigma_k}\leq R_{\sigma_k}(\Phi_1)=(\sigma_k-\sigma^*)\dfrac{\int_\dOmega \kappa(\Phi_1)_n^2}{\bigg(\int_\Omega g(x)|\Phi_1|^{p+1}\bigg)^\frac{2}{p+1}}\rightarrow 0$$
	as $k\rightarrow+\infty$. Moreover, since $u_k$ is a ground state for $J_{\sigma_k}$, $\|\Delta u_k\|_2^2-(1-\sigma_k)\int_\dOmega \kappa(u_k)_n^2=\int_\Omega g(x)|u_k|^{p+1}$ and, since $R_{\sigma_k}(u_k)=\Sigma_{\sigma_k}$, we deduce
	$$\bigg(\int_\Omega g(x)|u_k|^{p+1}\bigg)^\frac{p-1}{p+1}=\Sigma_{\sigma_k}\rightarrow 0.$$
	Hence, if $p>1$, $\int_\Omega g(x)|u_k|^{p+1}\rightarrow0$; otherwise, if $p\in(0,1)$, then $\int_\Omega g(x)|u_k|^{p+1}\rightarrow+\infty$, which implies, by $\Holder$ inequality as $g\in L^1(\Omega)$, that $\|u_k\|_\infty\rightarrow+\infty$.\\
	\noindent We have now to prove that, if $p>1$, this convergence to 0 is actually in the natural norm $H^2(\Omega)$. By Lemma \ref{eqnormnew}, $\|\cdot\|_{H_{\sigma_k}}$ is a norm in $H^2(\Omega)\cap\Ho$ for every $k$, so we are able to decompose in that norm the Hilbert space as $H^2(\Omega)\cap\Ho=span(\Phi_1)\oplus[span(\Phi_1)]^\perp$. Thus, for every $k$ there exist a unique $\alpha_k\in\R$ and $\psi_k\in[span(\Phi_1)]^\perp$ such that $u_k=\alpha_k\Phi_1+\psi_k$.\\
	\noindent Hence, for $k$ large enough,
	\begin{equation}\label{prodscalHsigma}
	\begin{split}
	o(1)&\geq\int_\Omega g(x)|u_k|^{p+1}=\|\Delta u_k\|_2^2-(1-\sigma_k)\int_\dOmega \kappa(u_k)_n^2=(u_k,u_k)_{H_{\sigma_k}}\\
	&=\alpha_k^2(\Phi_1,\Phi_1)_{H_{\sigma_k}}+(\psi_k,\psi_k)_{H_{\sigma_k}}.
	\end{split}
	\end{equation}
	First of all, 
	\begin{equation}\label{Phi1}
	(\Phi_1,\Phi_1)_{H_{\sigma_k}}=\|\Delta\Phi_1\|_2^2-(1-\sigma_k)\int_\dOmega \kappa(\Phi_1)_n^2=(\sigma_k-\sigma^*)\int_\dOmega \kappa(\Phi_1)_n^2.
	\end{equation}
	Moreover, denoting by $\tilde{\delta_2}(\Omega)$ the second eigenvalue of the Steklov problem, i.e.
	$$\tilde{\delta_2}(\Omega)=\inf_{span(\Phi_1)^\perp\setminus\{0\}}\dfrac{\|\Delta v\|_2^2}{\int_\dOmega \kappa v_n^2},$$
	and defining $\sigma^{**}:=1-\tilde{\delta_2}(\Omega)$, we get
	$$\|\Delta\psi_k\|_2^2\geq(1-\sigma^{**})\int_\dOmega \kappa(\psi_k)_n^2,$$
	from which
	\begin{equation}\label{psik}
	(\psi_k,\psi_k)_{H_{\sigma_k}}=\|\Delta\psi_k\|_2^2-(1-\sigma_k)\int_\dOmega \kappa(\psi_k)_n^2\geq \|\Delta\psi_k\|_2^2-\dfrac{1-\sigma_k}{1-\sigma^{**}}\|\Delta\psi_k\|_2^2=\dfrac{\sigma_k-\sigma^{**}}{1-\sigma^{**}}\|\Delta\psi_k\|_2^2.
	\end{equation}
	As a result, combining \eqref{prodscalHsigma} with \eqref{Phi1} and \eqref{psik}, we get:
	\begin{equation*}
	o(1)\geq\int_\Omega g(x)|u_k|^{p+1}=\alpha_k^2(\sigma_k-\sigma^*)\int_\dOmega \kappa(\Phi_1)_n^2+\dfrac{\sigma_k-\sigma^{**}}{1-\sigma^{**}}\|\Delta\psi_k\|_2^2.
	\end{equation*}
	Since we proved in Proposition \ref{eigenfunction} that the first Steklov eigenfunction is simple, we have $\sigma^{**}<\sigma^*$ and, recalling that $\sigma_k>\sigma^*$ by assumption, necessarily $\|\Delta\psi_k\|_2\rightarrow 0$. Hence,
	\begin{equation*}
	\begin{split}
	\int_\Omega g(x)|\alpha_k\Phi_1|^{p+1}&\leq\int_\Omega g(x)[|u_k|+|\psi_k|]^{p+1}\leq 2^p\int_\Omega g(x)[|u_k|^{p+1}+|\psi_k|^{p+1}]\\
	&\leq2^p\int_\Omega g(x)|u_k|^{p+1}+C^{p+1}(\Omega)\|g\|_1\|\psi_k\|_{H^2(\Omega)}\rightarrow0.
	\end{split}
	\end{equation*}
	As a result, $\alpha_k\rightarrow 0$ and we finally obtain
	$$\|u_k\|_{H^2(\Omega)}\leq|\alpha_k|\|\Phi_1\|_{H^2(\Omega)}+\|\psi_k\|_{H^2(\Omega)}\rightarrow 0.$$
\end{proof}

	If we read carefully the proof of Theorem \ref{approachingsigma*}, we notice that the fact that each $u_k$ is a ground state for $J_\sigma$ was necessary only to deduce that $\int_\Omega g(x)|u_k|^{p+1}\rightarrow0$, while to prove the convergence to 0 in $H^2(\Omega)$ norm it was only sufficient that each $u_k$ is a critical point (actually, an element of the Nehari manifold $\Ne_{\sigma_k}$, since the only step of the proof involved is \eqref{prodscalHsigma}). Consequently, we can directly state the following lemma, which will be useful when we will look at the radial case in Section 7:
\begin{lem}\label{approachingsigmastar}
	Let $(u_k)_{k\in\N}$ be a sequence of critical points of $J_{\sigma_k}$ in the superlinear case, such that $\int_\Omega g(x)|u_k|^{p+1}\rightarrow0$ as $\sigma_k\searrow\sigma^*$. Then $\|u_k\|_{H^2(\Omega)}\rightarrow0$.
\end{lem}

\section{Beyond the physical bounds: $\sigma>1$}
As briefly announced at the beginning of the previous section, here we want to investigate the behaviour of the ground states of $J_\sigma$ when $\sigma>1$. We assume again hereafter that $\Omega\subset\R^2$ is a bounded convex domain with $C^{1,1}$ boundary and \eqref{FNL} concerning the nonlinearity. As a consequence, the extension of the existence result is straightforward: in fact, in this case, by Lemma \ref{eqnormnew}, $\|\cdot\|_{H_\sigma(\Omega)}$ is still a norm on $H^2(\Omega)\cap\Ho$ and we can repeat the usual steps. Notice also that it is equivalent by these assumptions on $\Omega$ to consider critical points of $J_\sigma$ as far as weak solutions of the semilinear problem \eqref{PDEpSteklovg}.\\
\noindent The extension of positivity in this case seems not to be obvious, as already noticed in Remark \ref{extensionnotobvious}. We will provide here two different proofs (which will produce two slightly different results), the first one relying on the study of the convergence of ground states as $\sigma\rightarrow1$, which in the limit yields the Navier case, while the second is based on the method of dual cones by Moreau, connecting our semilinear problem with the linear one. We point out that the convergence result might be also of independent interest.\\
In the following, we will always consider the exponent of the nonlinearity \eqref{FNL} to be $p>1$, but similar results can be proved also in the sublinear framework (see Remarks \ref{rmksub} and \ref{rmksub2}).

\subsection{Convergence of ground states of $J_\sigma$ to ground states of $J_{NAV}$ as $\sigma\rightarrow1$}
In this section, $(u_k)_{k\in\N}$ will always denote a sequence of  ground states solutions of the Steklov problems
\begin{equation}\label{PDEpSteklov_k}
\begin{cases}
\Delta^2u=g(x)|u|^{p-1}u\quad&\mbox{in }\Omega,\\
u=\Delta u-(1-\sigma_k)\kappa u_n=0\quad&\mbox{on }\dOmega,
\end{cases}
\end{equation}
for a sequence $(\sigma_k)_{k\in\N}$ converging to $1$.
Moreover, in order to underline the peculiarity of the problem when $\sigma=1$, we indicate $J_{NAV}:=J_1$, whose critical points are the weak solution of the following Navier problem:
\begin{equation}\label{PDEpNav}
\begin{cases}
\Delta^2u=g(x)|u|^{p-1}u\quad&\mbox{in }\Omega,\\
u=\Delta u=0\quad&\mbox{on }\dOmega,
\end{cases}
\end{equation}
Finally, $\overline{u}$ will always denote a ground state of $J_{NAV}$. Our main result is to prove the convergence $u_k\rightarrow \ubar$ in the natural norm, i.e. in $H^2(\Omega)$, as $\sigma_k\rightarrow1$, no matter if $\sigma_k$ is less or greater than 1. First of all, a weaker result is enough:

\begin{lem}\label{fromweaktostrong}
	Let $(u_k)_{k\in\N}$ and $\ubar$ be as specified above. If $u_k\rightharpoonup\overline{u}$ weakly in $H^2(\Omega)$, then (up to a subsequence) $u_k\rightarrow\overline{u}$ strongly in $H^2(\Omega)$ as $\sigma_k\rightarrow1$.
\end{lem}
\begin{proof}
	As $u_k\rightharpoonup\overline{u}$ weakly in $H^2(\Omega)$, there exists $M>0$ such that $\|u_k\|_{H^2(\Omega)}^2\leq M$. Moreover, for each $k\in\N$, $u_k$ is a solution of \eqref{PDEpSteklov_k} and $\overline{u}$ of the Navier problem \eqref{PDEpNav}, thus, for every test function $\varphi\in H^2(\Omega)\cap\Ho$:
	\begin{equation}\label{soldebStek}
	\int_\Omega \Delta u_k\Delta\varphi-(1-\sigma_k)\int_\dOmega \kappa(u_k)_n\varphi_n=\int_\Omega g(x)|u_k|^{p-1}u_k\varphi,
	\end{equation}
	\begin{equation*}
	\int_\Omega \Delta\overline{u}\Delta\varphi=\int_\Omega g(x)|\overline u|^{p-1}\overline u\varphi.
	\end{equation*}
	Hence
	\begin{equation*}
	\begin{split}
	&C_A^{-1}\|u_k-\overline{u}\|_{H^2(\Omega)}^2\leq\|\Delta u_k-\Delta\overline{u}\|_2^2=\int_\Omega\Delta u_k\Delta(u_k-\overline{u})-\int_\Omega\Delta\overline{u}\Delta(u_k-\overline{u})\\
	&=(1-\sigma_k)\int_\dOmega \kappa (u_k)_n(u_k-\overline{u})_n+\bigg[\int_\Omega g(x)|u_k|^{p-1}u_k(u_k-\overline{u})-\int_\Omega g(x)|\overline u|^{p-1}\overline u(u_k-\overline{u})\bigg].
	\end{split}
	\end{equation*}
	For the first term:
	\begin{equation*}
	\begin{split}
	\bigg|(1-\sigma_k)\int_\dOmega \kappa (u_k)_n(u_k-\overline{u})_n\bigg|&\leq|1-\sigma_k|C_T^2\|\kappa\|_{L^\infty(\dOmega)}\|u_k\|_{H^2(\Omega)}\|u_k-\overline{u}\|_{H^2(\Omega)}\\
	&\leq|1-\sigma_k|C_T^2\|\kappa\|_{L^\infty(\dOmega)}M(M+\|\overline{u}\|_{H^2(\Omega)})\rightarrow 0,
	\end{split}
	\end{equation*}
	where $C_T$ is the constant in the Trace Theorem. Concerning the second, it is enough to invoke the Dominated Convergence Theorem as we have pointwise convergence and since
	\begin{equation*}
	\bigg|g(x)(|u_k|^{p-1}u_k-|\ubar|^{p-1}\ubar)(u_k-\ubar)\bigg|\leq |g(x)|[C(\Omega)^pM^p+|\ubar|^p][C(\Omega)M+\ubar]\in L^1(\Omega),
	\end{equation*}
	where $C(\Omega)$ is the constant in the embedding $H^2(\Omega)\hookrightarrow L^\infty(\Omega)$.
\end{proof}
\begin{remark}\label{fromweaktostrongremark}
	This result holds not only for ground states, but for generic solutions, i.e. if $(u_k)_{k\in\N}$ is a sequence of weak solutions of the Steklov problem \eqref{PDEpSteklov_k} and $\overline u$ a weak solution of the Navier problem \eqref{PDEpNav} and we know that $u_k\rightharpoonup\overline u$ weakly in $H^2(\Omega)$, then, up to a subsequence, it converges strongly too.
\end{remark}

\noindent A crucial observation is that the Nehari manifolds are nested with respect to the parameter $\sigma$:

\begin{lem}\label{t*disugrmk}
	If $\sigma_1<\sigma_2$ and fixing $u\in H^2(\Omega)\cap\Ho\setminus\{0\}$, then
	\begin{equation*}\label{t*disug}
	t_{\sigma_1}^*(u)\leq t_{\sigma_2}^*(u).
	\end{equation*}
\end{lem}
\begin{proof}
	In fact, $-(1-\sigma_1)<-(1-\sigma_2)$ and so
	\begin{equation*}
	t_{\sigma_1}^*(u)=\bigg(\dfrac{\int_\Omega(\Delta u)^2-(1-{\sigma_1})\int_\dOmega \kappa u_n^2}{\int_\Omega g(x)|u|^{p+1}}\bigg)^{\frac{1}{p-1}}\leq \bigg(\dfrac{\int_\Omega(\Delta u)^2-(1-{\sigma_2})\int_\dOmega \kappa u_n^2}{\int_\Omega g(x)|u|^{p+1}}\bigg)^{\frac{1}{p-1}}=t_{\sigma_2}^*(u).
	\end{equation*}
\end{proof}
	\noindent Notice that if $u\in H_0^2(\Omega)$ then one has the equality; if we suppose moreover that $\kappa>0$ a.e., we deduce also the converse.
	
\begin{prop}\label{ukH2bounded}
	The sequence of ground states $(u_k)_{k\in\N}$ is bounded in $H^2(\Omega)$.
\end{prop}
\begin{proof}
	Set $k_{max}$ such that $\sigma_{k_{max}}=\max\{(\sigma_k)_{k\in\N},1\}$ and so $u_{k_{max}}$ is a ground state for $J_{\sigma_{k_{max}}}$ (with the convention that if $\sigma_{k_{max}}$=1, then $u_{k_{max}}$ is a ground state for $J_{NAV}$).

	Defining $w_{k}:=t^*_{\sigma_k}(u_{k_{max}})u_{k_{max}}\in\Ne_{\sigma_k}$, that is, the "projection" of $u_{k_{max}}$ on the Nehari manifold $\Ne_{\sigma_k}$ along its half-line, one has 
	\begin{equation}\label{t^*disugcor}
	\int_\Omega g(x)|u_k|^{p+1}\leq \int_\Omega g(x)|w_{k}|^{p+1} \leq \int_\Omega g(x)|u_{k_{max}}|^{p+1}.
	\end{equation}
	Indeed, the first inequality comes from the fact that $u_k$ is a ground state of $J_{\sigma_k}$, which has the equivalent formulation \eqref{equivonNehari}; the second is obtained by Lemma \ref{t*disug} since
	\begin{equation*}
	\begin{split}
	\int_\Omega g(x)|w_k|^{p+1} &=(t^*_{\sigma_k}(u_{k_{max}}))^{p+1}\int_\Omega g(x)|u_{k_{max}}|^{p+1}\\
	&\leq(t^*_{\sigma_{k_{max}}}(u_{k_{max}}))^{p+1}\int_\Omega g(x)|u_{k_{max}}|^{p+1}=\int_\Omega g(x)|u_{k_{max}}|^{p+1}.
	\end{split}
	\end{equation*}
	Furthermore, for a generic  $\sigma>0$ (and here we can assume it without loss of generality),
	\begin{equation}\label{PariniAdolfsson}
	\int_\Omega(\Delta u)^2 - (1-\sigma)\int_\dOmega \kappa u_n^2 \geq \min\{\sigma,1\} C_A(\Omega)\|u\|_{H^2(\Omega)}^2.
	\end{equation}
	In fact, if $\sigma\in[1,+\infty)$ the proof is straightforward since $-(1-\sigma)\geq 0$, otherwise, if $\sigma\in(0,1)$:
	\begin{equation*}
	\begin{split}
	&\int_\Omega(\Delta u)^2 -(1-\sigma)\int_\dOmega\kappa u_n^2=\int_\Omega(\Delta u)^2+2(1-\sigma)\int_\Omega(-det(\nabla^2u))\\
	&=\int_\Omega \bigg[u_{xx}^2+u_{yy}^2+2\sigma u_{xx}u_{yy}+2(1-\sigma)u_{xy}^2\bigg]\geq\sigma\int_\Omega(\Delta u)^2+2(1-\sigma)\int_\Omega u_{xy}^2\\
	&\geq\sigma\int_\Omega(\Delta u)^2\geq\sigma C_A^{-1}(\Omega)\|u\|_{H^2(\Omega)}^2
	\end{split}
	\end{equation*}
	As a result, combining \eqref{t^*disugcor} with \eqref{PariniAdolfsson}, we get:
	$$\|u_k\|_{H^2(\Omega)}^2\leq\dfrac{C_A(\Omega)}{\min\{\sigma_k,1\}} \int_\Omega g(x)|u_k|^{p+1} \leq\dfrac{C_A(\Omega)}{\min\{\sigma_k,1\}}\int_\Omega g(x)|u_{k_{max}}|^{p+1},$$
	which is the estimate we needed.
\end{proof}

\noindent As a direct consequence of Proposition \ref{ukH2bounded}, the sequence $(u_k)_{k\in\N}$, up to a subsequence, is weakly convergent to some $u_\infty\in H^2(\Omega)\cap\Ho$ with strong convergence in $L^\infty(\Omega)$.
It is also easy to see that $u_\infty$ is a weak solution of the Navier problem \eqref{PDEpNav}: it is enough to	apply to \eqref{soldebStek} the weak convergence in $H^2(\Omega)$, the strong convergence in $L^2(\dOmega)$ of the normal derivatives and the Dominated Convergence Theorem. As a consequence, by Proposition \ref{fromweaktostrong}, the convergence $u_k\rightarrow u_\infty$ is strong in $H^2(\Omega)$.

\begin{thm}\label{convergence}
		Let	$\sigma_k\rightarrow1$ and $\Omega$ be a bounded convex domain in $\R^2$ with boundary of class $C^{1,1}$.
		Then the sequence $(u_k)_{k\in\N}$ of ground state solutions for the Steklov problems \eqref{PDEpSteklov_k} admits a subsequence $(u_{k_j})_{j\in\N}$ which converges in $H^2(\Omega)$ to $u_\infty$, which is a ground state for the Navier problem \eqref{PDEpNav}, thus strictly superharmonic.
\end{thm}
\begin{proof}
	Clearly, as $u_\infty$ is weak solution of \eqref{PDEpNav}, we have $J_{NAV}(u_\infty)\geq \inf_{\Ne_{NAV}}J_{NAV}$. Now we have to prove the converse inequality. Firstly, we have $J_{NAV}(u_\infty)\leq\liminf_{k\rightarrow+\infty}J_{\sigma_k}(u_k)$. Indeed,
	\begin{equation*}
	\begin{split}
	\liminf_{k\rightarrow+\infty}J_{\sigma_k}(u_k)&=\liminf_{k\rightarrow+\infty}\int_\Omega\dfrac{(\Delta u_k)^2}{2}-\lim_{k\rightarrow+\infty}\dfrac{1-\sigma_k}{2}\int_\dOmega \kappa (u_k)_n^2-\lim_{k\rightarrow+\infty}\int_\Omega\dfrac{g(x)|u_k|^{p+1}}{p+1}\\
	&\geq\int_\Omega\dfrac{(\Delta u_\infty)^2}{2}-\int_\Omega\dfrac{g(x)|u_\infty|^{p+1}}{p+1}=J_{NAV}(u_\infty),
	\end{split}
	\end{equation*}
	having used the compactness of the map  $\partial_n: H^1(\Omega)^2\rightarrow L^2(\Omega)$ and the Dominated Convergence Theorem. Moreover, if we suppose $\sigma_k<1$ for $k$ large enough, by Lemma \ref{t*disug} (with a similar argument than in \eqref{t^*disugcor}), for all $k\in \N$ we have 
	\begin{equation}\label{sigmapiccolo}
	J_{\sigma_k}(u_k)=\bigg(\dfrac{1}{2}-\dfrac{1}{p+1}\bigg)\int_\Omega g(x)|u_k|^{p+1}\leq\bigg(\dfrac{1}{2}-\dfrac{1}{p+1}\bigg)\int_\Omega g(x)|u_\infty|^{p+1}=J_{NAV}(u_\infty),
	\end{equation}
	so in this case we are done. If otherwise $\sigma_k>1$ for a infinite number of indexes,\eqref{sigmapiccolo} does not hold. In this case, without loss of generality, we can assume that $\sigma_k\searrow 1$. By the existence theorems in Section 2, we know that there exists $\ubar\in H^2(\Omega)\cap\Ho$ ground state for $J_{NAV}$ and define $\ubar_k:=t^*_{\sigma_k}(\ubar)\ubar$ the "projection" on the Nehari manifold $\Ne_{\sigma_k}$. Then $\|\ubar_k-\ubar\|_{H^2(\Omega)}=|1-t^*_{\sigma_k}(\ubar)|\|\ubar\|_{H^2(\Omega)}$ with
	\begin{equation*}
	1-(t^*_{\sigma_k}(\ubar))^{p-1}\stackrel{[\ubar\in\Ne_{NAV}]}{=}(t^*_{NAV}(\ubar))^{p-1}-(t^*_{\sigma_k}(\ubar))^{p-1}=2(1-\sigma_k)\dfrac{\int_\Omega det(\nabla^2\ubar)}{\int_\Omega g(x)|\ubar|^{p+1}}\rightarrow0,
	\end{equation*} 
	so $\ubar_k\rightarrow\ubar$ in $H^2(\Omega)$, which implies
	\begin{equation}\label{convergenceproof1}
	\int_\Omega g(x)|\ubar_k|^{p+1}\rightarrow\int_\Omega g(x)|\ubar|^{p+1}.
	\end{equation}
	Nevertheless, since $u_k$ is a ground state of $J_{\sigma_k}$,
	\begin{equation}\label{convergenceproof2}
	\int_\Omega g(x)|\ubar_k|^{p+1}\stackrel{[\ubar_k\in\Ne_{\sigma_k}]}{=}\bigg(\dfrac{1}{2}-\dfrac{1}{p+1}\bigg)J_{\sigma_k}(\ubar_k)\geq\bigg(\dfrac{1}{2}-\dfrac{1}{p+1}\bigg)J_{\sigma_k}(u_k)\stackrel{[u_k\in\Ne_{\sigma_k}]}{=}\int_\Omega g(x)|u_k|^{p+1};
	\end{equation}
	furthermore, since we assumed $\sigma_k>1$ and by Lemma \ref{t*disug},
	\begin{equation}\label{convergenceproof3}
	\begin{split}
	\int_\Omega g(x)|u_k|^{p+1}&\geq\int_\Omega g(x)|t^*_{NAV}(u_k)u_k|^{p+1}=\bigg(\dfrac{1}{2}-\dfrac{1}{p+1}\bigg)J_{NAV}(t^*_{NAV}(u_k)u_k)\\
	&\geq\bigg(\dfrac{1}{2}-\dfrac{1}{p+1}\bigg)J_{NAV}(\ubar)=\int_\Omega g(x)|\ubar|^{p+1}.
	\end{split}
	\end{equation}
	Combining \eqref{convergenceproof1}, \eqref{convergenceproof2} and \eqref{convergenceproof3}, we find that
	\begin{equation}\label{convergenceproof4}
	\int_\Omega g(x)|u_k|^{p+1}\rightarrow\int_\Omega g(x)|\ubar|^{p+1},
	\end{equation}
	from which $J_{\sigma_k}(u_k)\rightarrow J_{NAV}(\ubar)$, which completes our equality.\\
	To conclude, notice that we have already obtained in the proof of Proposition \ref{positivitysuperlinear} that ground states of the Navier problem \ref{PDEpNav} are strictly superharmonic.
\end{proof}

\begin{remark}\label{rmksub}
	The same analysis may be adapted also for the sublinear case $p\in(0,1)$, paying attention to some minor changes: for instance, Lemma \ref{t*disug} holds with the reverse inequality, but this compensates with the fact that this time the coefficient $\frac{1}{2}-\frac{1}{p}$ in the equivalent formulation of $J_{\sigma}$ is negative.
\end{remark}

\subsection{Regularity of solutions and $W^{2.q}$ convergence of ground states}
The convergence result of the previous section will be used to derive positivity of ground states when $\sigma$ lies in a right neighborhood of 1. Nevertheless, we will need a $C^{0,1}$ convergence to be able to control the normal derivatives on the boundary, thus we have to upgrade our convergence to a stronger norm. The first step will be to investigate, for a fixed $\sigma>\sigma^*$, the regularity of solutions of \eqref{PDEpSteklovg} and \eqref{PDEpNav} with just a slightly more regular boundary (actually, we will have to impose that $\dOmega$ is of class $C^2$). This will be obtained by means of the following lemma by Gazzola, Grunau and Sweers, which follows from a result by Agmon, Douglis and Nirenberg \cite[Theorem 15.3', p.707]{ADN}:

\begin{lem}[\cite{GGS}, Corollary 2.23]\label{GGSestimate}
	Let $q>1$ and take an integer $m\geq 4$. Assume that $\dOmega\in C^m$ and $a\in C^{m-2}$, then there exists $C=C(m,q,a,\Omega)>0$ such that
	$$\|u\|_{W^{m,q}(\Omega)}\leq C\bigg(\|u\|_q+\|\Delta^2u\|_{W^{m-4,q}(\Omega)}+\|u\|_{W^{m-\frac{1}{q},q}(\dOmega)}+\|\Delta u-au_n\|_{W^{m-2-\frac{1}{q},q}(\dOmega)}\bigg),$$
	for every $u\in W^{m,q}(\Omega)$. The same statement holds for any $m\geq 2$ provided the norms in the right hand side are suitably interpreted.
\end{lem}

\noindent Hence we have to define $\Delta^2u$ as a distribution in $W^{-2,q}(\Omega)$, i.e. acting on functions in $W^{2,q'}_0(\Omega)$. Let $u\in H^2(\Omega)\cap\Ho$ be a weak solution of \eqref{PDEpSteklovg}; we define the following linear functional over $H^2(\Omega)$:
\begin{equation*}\label{Delta2def}
\Delta^2u:H^2(\Omega)\ni\varphi\mapsto<\Delta^2u,\varphi>:=\int_\Omega\Delta u\Delta\varphi
\end{equation*}
which is well-defined and continuous. If we let 
\begin{equation*}\label{u^pdef}
u^p_g:\varphi\mapsto<u^p_g,\varphi>:=\int_\Omega g(x)|u|^{p-1}u\varphi,
\end{equation*}
it is clearly well-defined and continuous on $W^{2,q'}_0(\Omega)$ and, by the weak formulation of the PDE, on the subset $H^2_0(\Omega)$ it acts identically as $\Delta^2u$. As a result, we define 
\begin{equation}\label{Delta2def2}
\Delta^2u:W^{2,q'}_0(\Omega)\ni\varphi\mapsto<\Delta^2u,\varphi>:=\int_\Omega g(x)|u|^{p-1}u\varphi.
\end{equation}

\begin{prop}\label{regularity}
	If $\dOmega\in C^2$, for every $\sigma>\sigma^*$ the weak solutions of Steklov and Navier problems \eqref{PDEpSteklov_k}, \eqref{PDEpNav} lie in $W^{2,q}(\Omega)$ for every $q>2$.
\end{prop}
\begin{proof}
	Let $u\in H^2(\Omega)\cap\Ho$ be a weak solution of \eqref{PDEpSteklovg}. Applying Lemma \ref{GGSestimate} with $m=2$ and $a=(1-\sigma)\kappa\in\mathcal{C}^0(\dOmega)$ ($a=0$ for the Navier case), we find:
	$$\|u\|_{W^{2,q}(\Omega)}\leq C(q,\sigma,\Omega)\bigg(\|u\|_q+\|\Delta^2u\|_{W^{-2,q}(\Omega)}\bigg),$$
	which is well-defined in view of \eqref{Delta2def2}. Since 
	\begin{equation}\label{Delta2norm}
	\|\Delta^2u\|_{W^{-2,q}(\Omega)}=\sup_{0\neq\varphi\in W^{2,q'}_0(\Omega)}\dfrac{\bigg|\int_\Omega g(x)|u|^{p-1}u\varphi\bigg|}{\|\varphi\|_{W^{2,q'}_0(\Omega)}}\leq C(p,q,\Omega)\|g\|_1\|u\|_{H^2(\Omega)}^p,
	\end{equation}	
	we finally deduce from \eqref{Delta2norm}:
	$$\|u\|_{W^{2,q}(\Omega)}\leq C(q,\sigma,\Omega)\bigg(\|u\|_q+C(p,q,\Omega)\|g\|_1\|u\|_{H^2(\Omega)}^p\bigg)<+\infty.$$
\end{proof}

	\noindent We stress that we did not use either the fact that $u$ is a ground state solution, or its positivity: the above result holds true for every weak solution of Steklov and Navier problems.\\
	\noindent Let us now recall the following interpolation result:

\begin{lem}[Interpolation of Fractional Sobolev Spaces, \cite{BM}, Corollary 2]\label{interp}
		For $0\leq s_1<s_2<+\infty$, $1<p_1,p_2<+\infty$, for every $s,p$ such that $s=\theta s_1+(1-\theta)s_2$ and $\frac{1}{p}=\frac{\theta}{p_1}+\frac{1-\theta}{p_2}$, we have
		$$\|f\|_{W^{s,p}(\R^N)}\leq C\|f\|_{W^{s_1,p_1}(\R^N)}^\theta\|f\|_{W^{s_2,p_2}(\R^N)}^{1-\theta}.$$
\end{lem}

\begin{prop}\label{convergence(p)migliorata}
	Let $\Omega$ of class $C^2$ and $(u_k)_{k\in\N}$ be a sequence of weak solutions for the Steklov problems \eqref{PDEpSteklov_k} converging in $H^2(\Omega)$ to $\overline u$, weak solution for the Navier problem \eqref{PDEpNav}. Then the convergence is in $W^{2,q}(\Omega)$ for every $q\geq 2$.
\end{prop}
\begin{proof}
	Let $q\geq 2$ and apply the regularity estimate \eqref{GGSestimate} to $u_k-\overline u$ with $m=2$, $a=0$:
	\begin{equation}\label{Star}
	\|u_k-\overline u\|_{W^{2,q}(\Omega)}\leq C(q,\Omega)\bigg(\|u_k-\overline u\|_q+\|\Delta^2u_k-\Delta^2\overline u\|_{W^{-2,q}(\Omega)}+|1-\sigma_k|\|\kappa(u_k)_n\|_{W^{-\frac{1}{q},q}(\dOmega)}\bigg),
	\end{equation}
	since on $\dOmega$ we have $\Delta(u_k-\overline u)-a(u_k-\overline u)_n=\Delta u_k-\Delta\overline u=(1-\sigma_k)\kappa(u_k)_n$.\\
	By \eqref{Delta2def2} and Dominated Convergence Theorem:
	$$\|\Delta^2u_k-\Delta^2\overline u\|_{W^{-2,q}(\Omega)}=\sup_{0\neq\varphi\in W^{2,q'}_0(\Omega)}\dfrac{\bigg|\int_\Omega g(x)|u_k|^{p-1}u_k\varphi-\int_\Omega g(x)|\overline u|^{p-1}\overline u\varphi\bigg|}{\|\varphi\|_{W^{2,q'}_0(\Omega)}}\rightarrow0,$$
	similarly to \eqref{Delta2norm}. 
	We need now to prove that $(\kappa(u_k)_n)_{k\in\N}$ is bounded in $W^{-\frac{1}{q},q}(\dOmega)$. Notice that if we provide a uniform bound in $L^q(\dOmega)$, then we are done. In fact $W^{-\frac{1}{q},q}(\dOmega):=W^{\frac{1}{q},q'}(\dOmega)^*$ and $W^{\frac{1}{q},q'}(\dOmega)\hookrightarrow L^{q'}(\dOmega)$, so we directly infer $W^{-\frac{1}{q},q}(\dOmega)\hookleftarrow L^q(\dOmega)$.\\
	\noindent Moreover, it is known that, with our assumptions on $\dOmega$, the normal trace of functions in $W^{s,p}(\Omega)$ lies in $\Lp(\dOmega)$, provided $s>1+\frac{1}{p}$ (for this and some further sharper results, see \cite[Theorem 2]{Marsch}).	
	Hence,
	\begin{equation}\label{chain}
	\begin{split}
	\|\kappa(u_k)_n\|_{W^{-\frac{1}{q},q}(\dOmega)}&\leq C(q,\Omega)\|\kappa(u_k)_n\|_{\Lq(\dOmega)}\leq C(q,\Omega)\|\kappa\|_{L^\infty(\dOmega)}\|(u_k)_n\|_{\Lq(\dOmega)}\\
	&\leq C(q,\Omega,s)\|\kappa\|_{L^\infty(\dOmega)}\|u_k\|_{W^{s,q}(\Omega)},
	\end{split}
	\end{equation}
	for some $s>1+\frac{1}{q}$. Thus, we need to find an appropriate fractional Sobolev spaces in which $H^2(\Omega)$ is embedded.
	We claim that $H^2(\Omega)\hookrightarrow W^{1+3/2q,q}(\Omega)$. Actually, it is enough to prove that $H^1(\Omega):=W^{1,2}(\Omega)\hookrightarrow W^{3/2q,q}(\Omega)$ by definition of $W^{s,p}(\Omega)$ for $s>1$.
	So, let $u\in W^{1,2}(\Omega)$; by Stein total extension theorem \cite[Theorem 5.24]{A} there exists $U\in W^{1,2}(\R^2)$ such that $U_{|_\Omega}=u$ a.e. and $\|U\|_{W^{1,2}(\R^2)}\leq C\|u\|_{W^{1,2}(\Omega)}$ for some positive constant independent of $u$. Applying the interpolation result Lemma \ref{interp} to $U$ with $\theta=\frac{3}{2q}$ and the Sobolev embedding $W^{1,2}(\R^2)\hookrightarrow L^{4q-6}(\R^2)$ since $4q-6\geq2$:
	\begin{equation*}
	\|U\|_{W^{3/2q,q}(\R^2)}\leq C\|U\|_{W^{1,2}(\R^2)}^\frac{3}{2q}\|U\|_{L^{4q-6}(\R^2)}^{1-\frac{3}{2q}}\leq C_1\|U\|_{W^{1,2}(\R^2)}.
	\end{equation*}
	Hence,
	\begin{equation*}\label{interpol}
	\|u\|_{W^{3/2q,q}(\Omega)}=\|U\|_{W^{3/2q,q}(\Omega)}\leq\|U\|_{W^{3/2q,q}(\R^2)}\leq C_1\|U\|_{W^{1,2}(\R^2)}\leq C_2\|u\|_{W^{1,2}(\Omega)}.
	\end{equation*}
	As a result, noticing that $s=1+\frac{3}{2q}>1+\frac{1}{q}$, we can continue $\eqref{chain}$, obtaining:
	\begin{equation*}
	\|\kappa(u_k)_n\|_{W^{-\frac{1}{q},q}(\dOmega)}\leq C(q,\Omega)\|\kappa\|_{L^\infty(\dOmega)}\|u_k\|_{W^{1+3/2q,q}(\Omega)}\leq\tilde{C}(q,\Omega)\|\kappa\|_{L^\infty(\dOmega)}\|u_k\|_{H^2(\Omega)},
	\end{equation*}
	which is uniformly bounded in $k$. Combining estimate \eqref{Star} with the ones above for the second and the third term of \eqref{Star}, we finally end up with the strong convergence in $W^{2,q}(\Omega)$.
\end{proof}

\subsection{Extending positivity, part 1: a convergence argument}
Let us start noticing that, by Morrey's embeddings, the convergence in $W^{2,q}(\Omega)$ for every $q\geq 2$ of Proposition \ref{convergence(p)migliorata} implies the convergence in $C^{1,\alpha}(\Omegabar)$ for every $\alpha<1$ (thus in particular in $C^{1}(\Omegabar)$). This will be the main ingredient in the next proof.

\begin{prop}\label{positivity1}
	Let $\Omega\subset\R^2$ be a bounded convex domain of class $C^2$ and $(\sigma_k)_{k\in\N}$ be a sequence of parameters such that $\sigma_k\searrow 1$ and $(u_k)_{k\in\N}$ a sequence of ground states for the functional $J_{\sigma_k}$. Then there exists a subsequence $(u_{k_j})_{j\in\N}$ and $j_0\in\N$ such that $u_{k_j}>0$ in $\Omega$ for every $j\geq j_0$.
\end{prop}
\begin{proof}
	By Propositions \ref{regularity} and \ref{convergence(p)migliorata} and by the previous observation, we know that, up to a subsequence, $u_k\rightarrow\overline u$ in $C^{1}(\Omegabar)$ for some $\overline{u}$, ground state for $J_{NAV}$.\\
	\noindent Since $\Omega$ has a $C^2$ boundary, the interior sphere condition holds and one can extend the outer normal vector $n$ in a small neighborhood $\omega_0\subset\Omega$ of $\dOmega$ and thus define here $\ubar_n:=\nabla\ubar\cdot n$ (see \cite[Chapter 4]{Sperb}). Moreover, since $\overline u$ is strictly superharmonic, the normal derivative $\overline u_n$ is negative on $\dOmega$ and, by compactness of $\dOmega$ and continuity of $\overline u_n$, there exists $\alpha>0$ such that
	$${\overline u_n}_{|_\dOmega}\leq-\alpha<0.$$
	Hence, again by continuity, there exists a second neighborhood $\omega\subset\omega_0$ of $\dOmega$ such that
	$${\overline u_n}_{|_{\omega}}\leq-\dfrac{2}{3}\alpha<0.$$
	Take now $\varepsilon_1=\frac{\alpha}{3}$: by the $C^1(\Omegabar)$ convergence, there exists $k_1\in\N$ such that for every $k\geq k_1$ and $x\in\omega$:
	\begin{equation*}
	\begin{split}
	|(u_k)_n(x)|&\geq
	|\overline u_n(x)|-|(u_k)_n(x)-\overline u_n(x)|\\
	&>\dfrac{2\alpha}{3}-\||n|\|_{L^\infty(\omega)}\||\nabla u_k-\nabla\overline u|\|_{L^\infty(\Omega)}>\dfrac{2\alpha}{3}-\varepsilon_1>\dfrac{\alpha}{3}.
	\end{split}
	\end{equation*}
	By the interior sphere condition, the map $\omega\ni x\mapsto x_0\in\dOmega$ such that $d(x,x_0)=\inf\{d(x,y)|y\in\dOmega\}$ is well defined and the vector $x-x_0$ has the same direction as $n(x)$ and $n(x_0)$. Hence by Lagrange Theorem and recalling that ${u_k}_{|_\dOmega}=0$, for $x\in\omega$:
	\begin{equation}\label{pos1}
	|u_k(x)|=|u_k(x)-u_k(x_0)|\geq\min_{y\in[x_0,x]}|(u_k)_n(y)||x-x_0|>\dfrac{\alpha}{3}|x-x_0|>0.
	\end{equation}
	Moreover, notice that by compactness of $\Omega_0:=\Omega\setminus\omega$, the remaining part of $\Omega$,
	$${\overline u}_{|_{\Omega_0}}\geq\min_{\Omega_0}\overline u:=m>0$$
	and so by the uniform convergence it is easy to deduce that, for $k$ large enough, $u_k(x)>\frac{m}{2}$ for every $x\in\Omega_0$.
	The results follows combining this with \eqref{pos1}.
\end{proof}

\begin{thm}\label{sigma1}
	Let $\Omega\subset\R^2$ be a bounded convex domain of class $C^2$; then there exists $\sigma_1>1$ such that for every $\sigma\in(1,\sigma_1)$ the ground states of $J_\sigma$ are positive in $\Omega$.
\end{thm}
\begin{proof}
	By contradiction, suppose that such $\sigma_1$ does not exist. Hence we would be able to find a sequence $(\sigma_k)\searrow 1$ such that for each of them there exists a ground state $u_k$ for $J_{\sigma_k}$ which is not positive. This would contradict the Proposition \ref{positivity1}.
\end{proof}
\begin{remark}\label{remarkstrictpositivity}
	As we are dealing with continuous functions, since $H^2(\Omega)\hookrightarrow C^0(\Omegabar)$, we are interested in the strict positivity \textit{everywhere in $\Omega$} and not only a.e. in $\Omega$. Theorem \ref{sigma1} gives a positive answer for this question: in fact, as $\ubar\in H^2(\Omega)=W^{2,N}(\Omega)$ is strictly superharmonic, by strong maximum principle for strong solutions \cite[Theorem 9.6]{GT}, we deduce that it cannot achieve its minimum on the interior of $\Omega$, thus $\ubar(x)>0$ \textit{for every} $x\in\Omega$. By the $C^1$ convergence we deduce the same strict inequality for $u_\sigma$, with $\sigma\in(1,\sigma_1)$.
\end{remark}

\subsection{Extending positivity, part 2: Moreau dual cones decomposition}
Our aim is to investigate a further extension of the positivity result found in Theorem \ref{sigma1}, possibly for the whole range $\sigma\in(1,+\infty)$. It seems natural if we think to the following fact: similarly to what already obtained for the Navier problem, one can prove the convergence in $H^2(\Omega)$, as $\sigma\rightarrow+\infty$, of a sequence of ground states of $J_\sigma$ to a least-energy solution of the Dirichlet problem
\begin{equation}\label{dir}
\begin{cases}
\Delta^2u=g(x)|u|^{p-1}u\quad&\mbox{in }\Omega\\
u=u_n=0\quad&\mbox{on }\dOmega,
\end{cases}
\end{equation}
at least when $\kappa$ is positive a.e. on $\dOmega$.
Since we already know that in some cases the ground states of \eqref{dir} are positive (for instance if $\Omega$ is a ball, see \cite{FGW}), we expect to be able to completely extend positivity for such domains.\\
After a brief explanation of the convergence just mentioned above, we will apply Moreau's method of dual cones to infer the intervals of positivity for the semilinear problem. At the end, one may also compare the resulting analysis with the respective one for the linear problem with the same boundary conditions, due to Gazzola and Sweers in \cite{GS}.

\subsubsection{The Dirichlet Problem}
The argument is similar of what we used in Section 6.1 for the convergence to the Navier problem, but now we have to pay attention to the fact that in this case the two functional spaces are different ($H^2(\Omega)\cap\Ho$ for the Steklov problem and $H^2_0(\Omega)$ for the Dirichlet). We are not giving here the details of the proof of the existence of ground states of \eqref{dir}, as it can be obtained as for the Steklov framework by the Nehari method of Section 3.
In the following, we assume $\Omega$ to be a bounded convex domain in $\R^2$ with boundary of class $C^{1,1}$ and $\sigma>1$.
We suppose also that the curvature $k$ is positive a.e, that is $\dOmega$ has not flat parts.
Moreover, as usual, $u_k$ will always denote a ground state for $J_{\sigma_k}$ and $\overline u$ a ground state for $J_{DIR}: H^2_0(\Omega)\rightarrow\R$ defined as
$$J_{DIR}(u)=\dfrac{1}{2}\int_\Omega(\Delta u)^2-\dfrac{1}{p+1}\int_\Omega g(x)|u|^{p+1},$$
whose critical points are weak solutions of $\eqref{dir}$. Moreover, as in the Steklov case, we define the Nehari manifold for $J_{DIR}$:
$$\Ne_{DIR}:=\{u\in H^2_0(\Omega)\setminus\{0\} \,|\, J_{DIR}'(u)[u]=0\}.$$
First of all, notice that, by definition of $J_\sigma$, for each $\sigma$,
\begin{equation}\label{t*Hoo}
{J_{\sigma}}_{|_{H^2_0(\Omega)}}=J_{DIR},
\end{equation}
so $\Ne_\sigma$ restricted to the subspace $H^2_0(\Omega)$ coincides with $\Ne_{DIR}$.

\begin{thm}\label{towardsDirichlet}
	Let	$\sigma_k\rightarrow+\infty$ and $\Omega$ be a bounded convex domain in $\R^2$ with boundary of class $C^{1,1}$. Assume also that the curvature $k$ is positive a.e on $\dOmega$.
	Then the sequence $(u_k)_{k\in\N}$ of ground state of $(J_{\sigma_k})_{k\in\N}$ admits a subsequence $(u_{k_j})_{j\in\N}$ convergent in $H^2(\Omega)$ to $\overline{u}$, which is a ground state for the Dirichlet problem \eqref{dir}.
\end{thm}
\begin{proof}
	We follow the same steps as in Section 6.1 to deduce Theorem \ref{convergence}. Firstly, we prove that $(u_k)_{k\in\N}$ is bounded in $H^2(\Omega)$. Indeed, fix $\overline w\in\Hoo$ a ground state for the Dirichlet problem $\eqref{dir}$, then
	\begin{equation}\label{dirbound}
	\begin{split}
	\|\Delta u_k\|_2^2&\leq\int_\Omega(\Delta u_k)^2-(1-\sigma_k)\int_\dOmega \kappa(u_k)_n^2=\int_\Omega g(x)|u_k|^{p+1}=\inf_{v\in\Ne_{\sigma_k}}\int_\Omega g(x)|v|^{p+1}\\
	&\leq\inf_{v\in\Ne_{\sigma_k}\cap H^2_0(\Omega)}\int_\Omega g(x)|v|^{p+1}=\int_\Omega g(x)|\overline w|^{p+1}.
	\end{split}
	\end{equation}
	Hence, there exists $\overline{u}\in H^2(\Omega)\cap\Ho$ such that, up to a subsequence, $u_k\rightharpoonup\overline u$ weakly in $H^2(\Omega)$. Moreover, $\eqref{dirbound}$ implies that
	$$0\leq (\sigma_k-1)\int_\dOmega \kappa(u_k)_n^2\leq \int_\Omega g(x)|u_k|^{p+1}\leq C(\Omega,p)\|g\|_1\|u_k\|_{H^2(\Omega)}^{p+1}\leq D(\Omega,p,g)$$
	and, taking into account that $\sigma_k\rightarrow+\infty$, we deduce that
	$$\int_\dOmega\kappa(u_k)_n^2\rightarrow 0.$$
	Furthermore, by the compactness of the map $\partial_n: H^2(\Omega)\rightarrow L^2(\dOmega)$, we have also that
	$$\int_\dOmega\kappa(u_k)_n^2\rightarrow \int_\dOmega \kappa\overline{u}_n^2.$$
	Hence, combining the two and recalling that we assumed $\kappa>0$ on $\dOmega$, we deduce that $\overline{u}_n\equiv 0$ on $\dOmega$ and thus $\overline{u}\in\Hoo$.\\
	\noindent Finally, testing the weak formulation of problem \eqref{PDEpSteklov_k} with $\varphi\in\Hoo$ and passing to the limit as $k\rightarrow+\infty$, we deduce that
	$$\int_\Omega\Delta\overline{u}\Delta\varphi=\int_\Omega g(x)|\overline{u}|^{p-1}\overline{u}\varphi,$$
	so $\overline u$ is a solution of the Dirichlet problem \eqref{dir} and, similarly to Lemma \ref{fromweaktostrong}, we can prove that the convergence is strong in $H^2(\Omega)$.
	It remains to prove that $\ubar$ is actually a ground state for $J_{DIR}$. Let $\overline w\in\Hoo$ be a ground state solution of $J_{DIR}$. Then, by \eqref{t*Hoo}:
		\begin{equation*}
		m=J_{DIR}(\overline w)=J_{\sigma_k}(t^*_{\sigma_k}(\overline w)\overline w)\geq\inf_{\Ne_{\sigma_k}\cap H^2_0(\Omega)}J_{\sigma_k}\geq\inf_{\Ne_{\sigma_k}}J_{\sigma_k}=J_{\sigma_k}(u_k),
		\end{equation*}
		hence we deduce that $m\geq\liminf_{k\rightarrow+\infty}J_{\sigma_k}(u_k)$. Moreover, by strong convergence,
		$$J_{DIR}(\overline u)=\bigg(\dfrac{1}{2}-\dfrac{1}{p+1}\bigg)\int_\Omega g(x)|\overline u|^{p+1}=\lim_{k\rightarrow+\infty}\bigg(\dfrac{1}{2}-\dfrac{1}{p+1}\bigg)\int_\Omega g(x)|u_k|^{p+1}=\lim_{k\rightarrow+\infty}J_{\sigma_k}(u_k).$$
		Finally, since $\overline u$ is a solution of the Dirichlet problem \eqref{dir}, then $\overline{u}\in\Ne_{DIR}$, so:
		$$m\leq J_{DIR}(\overline{u})\leq \liminf_{k\rightarrow+\infty}J_{\sigma_k}(u_k)\leq m.$$
\end{proof}

\subsubsection{Moreau dual cones decomposition}
So far, we have proved the existence of ground states for the Dirichlet problem \eqref{dir} and the convergence result as $\sigma\rightarrow+\infty$. Proving positivity of ground states of \eqref{dir} is quite a hard subject, since it strongly relies to the geometry of the domain, even in the linear case, where $f(x,u)=f(x)$: we refer to \cite{SwRew} for a short survey. Anyway, there are some cases in which it holds: for instance, the Dirichlet problem in the ball has been studied in \cite{FGW}, which covers the case where $g\equiv1$, but whose arguments hold also in the general situation.\\
\noindent Our strategy is mainly inspired by this last work and it was firstly applied to fourth order problems by Gazzola and Grunau in \cite{GGr}. Briefly, we use Moreau decomposition in dual cones (for the original paper, see \cite{Moreau}) to obtain from a supposed sign-changing ground state solution $u$, a function $w$ of one sign and in the same space with a strictly lower energy level, leading to a contradiction.
In our case, in order to apply this machinery, we have to impose that the associated linear problem is positivity preserving: this will be the connection between the two problems.
\begin{defn}
	Let $\Omega\subset\R^2$ be a bounded domain of class $C^{1,1}$ and fix $\sigma\in\R$. The linear Steklov boundary problem
	\begin{equation}\label{PDElinSteklov}
	\begin{cases}
	\Delta^2u=f\quad&\mbox{in }\Omega,\\
	u=\Delta u-(1-\sigma)\kappa u_n=0\quad&\mbox{on }\dOmega,
	\end{cases}
	\end{equation}
	is \textit{positivity preserving} in $\Omega$ if there exists a unique solution $u\in H^2(\Omega)\cap\Ho$ and $f\geq0$ implies $u\geq0$, and this holds for each $f\in L^2(\Omega)$. We shortly say that "$\Omega$ is a [PPP$_\sigma$] domain for \eqref{PDElinSteklov}".
\end{defn}
\begin{defn}
	Let $H$ be a Hilbert space with scalar product $(\cdot,\cdot)_H$ and $K\subset H$ be a nonempty closed convex cone. Its \textit{dual cone} $K^*$ is defined as
	$$K^*:=\{w\in H\,|\,(w,v)_H\leq0,\,\,\forall v\in K\}.$$
\end{defn}

\begin{thm}[Moreau Dual Cone Decomposition, \cite{GGS}, Theorem 3.4]
	Let $H$ be a Hilbert space with scalar product $(\cdot,\cdot)_H$ and $K$ and $K^*$ as before. Then for every $u\in H$, there exists a unique couple $(u_1,u_2)\in K\times K^*$ such that $u=u_1+u_2$ and $(u_1,u_2)_H=0$.
\end{thm}

\noindent Our aim is to apply this result with $(H,\|\cdot\|_H)=(H^2(\Omega)\cap\Ho;\|\cdot\|_{H_\sigma})$, where $\|\cdot\|_{H_\sigma}$ is the norm \eqref{usualequivnorm}, and $K:=\{v\in H\,|\,v\geq0\}$, the cone of nonnegative functions, looking for a decomposition of each element of the space in positive and negative "parts". Hence we need a characterization of the dual cone $K^*$:

\begin{lem}\label{dualcone}
	If $\Omega$ is a [PPP$_\sigma$] domain for \eqref{PDElinSteklov} for a fixed $\sigma\in\R$, then $K^*\subseteq\{w\in H\,|\,w<0\,\,\mbox{a.e.}\}\cup\{0\}$.
\end{lem}
\begin{proof}
	We adapt here the proof of \cite[Proposition 3.6]{GGS}. Let $\varphi\in C^\infty_c(\Omega)$, $\varphi\geq0$ and let $v_\varphi\in H^2(\Omega)\cap\Ho$ the unique weak solution of the linear problem
	\begin{equation*}
	\begin{cases}
	\Delta^2v_\varphi=\varphi\quad&\mbox{in }\Omega\\
	v_\varphi=\Delta v_\varphi-(1-\sigma)\kappa (v_\varphi)_n=0\quad&\mbox{on }\dOmega,
	\end{cases}
	\end{equation*}
	that is, for every test function $w\in H^2(\Omega)\cap\Ho$, we have
	$$(v_\varphi,w)_{H_\sigma}:=\int_\Omega\Delta v_\varphi\Delta w-(1-\sigma)\int_\dOmega\kappa(v_\varphi)_n w_n=\int_\Omega \varphi w.$$
	Hence, suppose $w=u\in K^*$: as $\Omega$ is a [PPP$_\sigma$] domain and $\varphi\geq0$, we deduce that $v_\varphi\geq0$, so $v_\varphi\in K$ and thus $(v_\varphi,u)_{H_\sigma}\leq 0$. As a result, we have obtained that for every $\varphi\in C^\infty_c(\Omega)$, $\varphi\geq0$, $\int_\Omega\varphi u\leq 0$, which implies that $u\leq 0$ a.e. in $\Omega$.\\
	Moreover, let us suppose that the null-set of $u$, namely $N:=\{x\in\Omega\,|\,u(x)=0\}$, has positive measure, consider $\psi:=\chi_N\neq 0$ and let $v_0$ be the unique solution of the following linear Navier problem:
	\begin{equation}\label{v0lin}
	\begin{cases}
	\Delta^2v_0=\psi\quad&\mbox{in }\Omega,\\
	v_0=\Delta v_0=0\quad&\mbox{on }\dOmega.
	\end{cases}
	\end{equation}
	Then $v_0$ is strictly superharmonic by the maximum principle, thus $v_0>0$ and, by Hopf Lemma, $(v_0)_n<0$. As a result, for any function $v\in H^2(\Omega)\cap\Ho$ one can produce two positive constant $\alpha$, $\beta$ such that $v+\alpha v_0\geq 0$ and $v-\beta v_0\leq 0$. Moreover we claim that $(u,v_0)_{H_\sigma}\geq 0$. In fact, as $v_0$ is the weak solution of \eqref{v0lin} and by definition of $\psi$:
	$$\int_\Omega\Delta u\Delta v_0=\int_\Omega u\psi=\int_N u=0,$$
	thus, since $\sigma>1$, $\kappa\geq 0$, $u_n\leq 0$ as $u\geq 0$, and $(v_0)_n<0$:
	$$(u,v_0)_{H_\sigma}:=\int_\Omega\Delta u\Delta v_0-(1-\sigma)\int_\dOmega\kappa u_n(v_0)_n\geq 0.$$
	As a result, recalling that $u\in K^*$, $v+\alpha v_0\in K$ and  $v-\beta v_0\in(-K)$, we have the following chain of inequalities:
	\begin{equation*}
	0\geq(u,v+\alpha v_0)_{H_\sigma}=(u,v)_{H_\sigma}+\alpha(u,v_0)_{H_\sigma}\geq(u,v)_{H_\sigma}\geq(u,v)_{H_\sigma}-\beta(u,v_0)_{H_\sigma}=(u,v-\beta v_0)_{H_\sigma}\geq 0,
	\end{equation*}
	which implies that $(u,v)_{H_\sigma}=0$, and this holds for all $v\in H^2(\Omega)\cap\Ho$. Hence this is true also for $v$ defined as the unique solution of the following Steklov problem:
	\begin{equation}\label{vlin}
	\begin{cases}
	\Delta^2v=u\quad&\mbox{in }\Omega,\\
	v=\Delta v-(1-\sigma)\kappa v_n=0\quad&\mbox{on }\dOmega,
	\end{cases}
	\end{equation}
	and, using $u$ as a test function, we deduce that
	$$0=(u,v)_{H_\sigma}=
	\int_\Omega u^2=\|u\|^2_2,$$
	which implies $u=0$ a.e.
\end{proof}

\begin{prop}\label{possigma>1}
	Let $\sigma>1$ and suppose $\Omega$ is a [PPP$_\sigma$] domain for \eqref{PDElinSteklov}. Then the ground states of $J_\sigma$ are a.e. strictly of only one sign.
\end{prop}
\begin{proof}
	Let $u\in H^2(\Omega)\cap\Ho$ be such a ground state and suppose by contradiction that $u$ is sign-changing. Denoting as before  the cone of nonnegative functions by $K$, by Moreau Decomposition there exists a unique couple $(u_1,u_2)\in K\times K^*$ such that $u=u_1+u_2$ and $(u_1,u_2)_{H_\sigma}=0$. Hence we know that $u_1\geq0$ and, by the Lemma \ref{dualcone}, $u_2<0$. Moreover, $u$ is supposed to change sign, so $u_1\neq 0$.\\
	\noindent Defining $w:=u_1-u_2\in H^2(\Omega)\cap\Ho$, then $w>|u|$. Indeed,
	$$w=u_1-u_2> u_1+u_2=u,$$
	$$w=u_1-u_2> -u_1-u_2=-u.$$
	Consequently, $\int_\Omega g(x)|w|^{p+1}>\int_\Omega g(x)|u|^{p+1}$ and, since the decomposition is orthogonal under that norm, $\|w\|_{H_\sigma}^2=\|u_1\|_{H_\sigma}^2+\|u_2\|_{H_\sigma}^2=\|u\|_{H_\sigma}^2$. Moreover, by Lemma \ref{t*lemma}, there exists $t^*:=t^*(w)\in (0,+\infty)$ such that $w^*:=t^*(w)w\in\Ne_\sigma$. Hence we deduce
	\begin{equation*}
	\begin{split}
	J_\sigma(w^*)&=\dfrac{(t^*)^2}{2}\|w\|_{H_\sigma}^2-\dfrac{(t^*)^{p+1}}{p+1}\int_\Omega g(x)|w|^{p+1}\\
	&<\dfrac{(t^*)^2}{2}\|u\|_{H_\sigma}^2-\dfrac{(t^*)^{p+1}}{p+1}\int_\Omega g(x)|u|^{p+1}=J_\sigma(t^*(w)u)\leq J_\sigma(u),
	\end{split}
	\end{equation*}
	since $u$ is the maximum of $J_\sigma$ on the half-line $\{tu\,|\,t\in(0,+\infty)\}$ by Lemma \ref{t*lemma}, thus a contradiction again, since $u$ was the infimum of $J_\sigma$ on the Nehari manifold $\Ne_\sigma$. Hence we infer that $u\geq 0$.\\
	\noindent Finally, as $u$ is a critical point of $J_\sigma$, we have for each a positive test function $\varphi\in H^2(\Omega)\cap\Ho$:
	$$(u,\varphi)_{H_\sigma}=\int_\Omega\Delta u\Delta\varphi-(1-\sigma)\int_\dOmega\kappa u_n\varphi_n=\int_\Omega g(x)u^p\varphi\geq 0,$$
	which implies that $-u\in K^*$. Applying now Lemma \ref{dualcone}, we get $-u<0$, i.e. $u>0$.
\end{proof}

\noindent As a consequence, the problem of proving positivity of ground state is led back to a problem of positivity preserving for the linear problem, which was already tackled and solved by Gazzola and Sweers in \cite{GS}.

\begin{thm}\label{possigma>1thm}
	Let $\sigma>1$ and $\Omega\subset\R^2$ be a bounded convex domain with $\dOmega$ of class $C^2$. There exists $\tilde\delta_c(\Omega)\in(1,+\infty]$ such that if $\sigma\in(1,\tilde\delta_c(\Omega))$, the ground states of the functional $J_\sigma$ are a.e. strictly of only one sign.
\end{thm}

\begin{proof}
	We follow the notation of \cite{GS}. Choosing $\beta=\kappa$ in (iii) of \cite[Theorem 4.1]{GS}, we infer the existence of $\delta_{c,\kappa}(\Omega)\in[-\infty,0)$ such that if $(1-\sigma)\kappa\geq\delta_{c,\kappa}(\Omega)\kappa$, then the positivity preserving for problem \eqref{PDElinSteklov} holds in $\Omega$. Hence, denoting $\tilde\delta_c(\Omega):=1+|\delta_{c,\kappa}(\Omega)|$, we can apply Proposition \ref{possigma>1}, provided $\sigma<\tilde\delta_c(\Omega)$.
\end{proof}
	
	\noindent Comparing Theorems \ref{possigma>1thm} and \ref{sigma1}, one may argue that we have nothing more than what we already knew: in both we obtain the existence of $\sigma_1=\sigma_1(\Omega)>1$ such that for all $\sigma\in(1,\sigma_1)$ the ground state solutions of problem \eqref{PDEpSteklovg} are positive. Nevertheless, in Theorem \ref{possigma>1thm} we get a further precise information about how the interval of positivity depends on the domain, relating it strongly with the positivity preserving property. This fact is striking in the case of the disk and allows us to finally answer to the question which opened the section.
	
\begin{cor}\label{thmBALL}
	Let $B\subset\R^2$ be a disk and let $\sigma>1$. Then the ground states of the functional $J_\sigma$ are a.e. strictly of only one sign.
\end{cor}
\begin{proof}
	It is enough to notice that here $\kappa=1$ and applying \cite[Theorem 2.9]{GS} one can deduce $\delta_{c,\kappa}(B)=-\infty$, which implies $\tilde\delta_c(B)=+\infty$.
\end{proof}

	\noindent One should finally notice that here the positivity found by the dual cones method is up to a subset of the domain with zero Lebesgue measure, so \textit{almost everywhere} in $\Omega$. This is the price we have to pay to extend the positivity beyond the parameter $\sigma_1$ found in Theorem \ref{sigma1} (cf. also Remark \ref{remarkstrictpositivity}).

\begin{remark}\label{rmksub2}
	Again, up to some easy modifications in the proofs, both the convergence in Theorem \ref{towardsDirichlet} and the positivity result in Theorem \ref{possigma>1thm} hold also in the sublinear case $p\in(0,1)$.
\end{remark}

\section{Radial Case}

This section is devoted to some further investigations when the domain is a disk in $\R^2$ and the function $g$ is radial, regarding existence, positivity and some qualitative properties of radially symmetric solutions. Moreover, we establish the counterpart of the convergence results of Sections 5 and 6, but for general radial positive solutions.\\
For simplicity, we focus on the problem
\begin{equation}\label{PDEpSteklovRAD}
\begin{cases}
\Delta^2u=g(x)|u|^{p-1}u\quad&\mbox{in }B,\\
u=\Delta u-(1-\sigma)u_n=0\quad&\mbox{on }\partial B,
\end{cases}
\end{equation}
where $B:=B_1(0)\subset\R^2$, $g=g(|x|)$ lies in $L^1(B)$ and it is strictly positive inside $B$. Moreover, we let the parameter $\sigma\in\R$ and $p\in(0,1)\cup(1,+\infty)$ to cover both the sublinear and the superlinear case. Notice that the curvature does not appear in the mixed boundary condition since $\kappa(B)\equiv1$.

\subsection{Positive radially decreasing solutions and global bounds}

First of all, by Proposition \ref{nonexistence}, our analysis concerns only the range $\sigma>-1$: in fact, if $\Omega=B$, one has $\sigma^*=-1$, since the first Steklov eigenvalue $\tilde\delta_1(B)=2$ (see \cite[Proposition 12]{BGM}).\\
\noindent Retracing exactly the same steps of Sections 3 and 4, it is quite easy to obtain the existence of a positive radial solution. In fact, confining ourselves to the closed subspace of radial functions
$$H_{rad}(B):=\{u\in H^2(B)\cap H^1_0(B)\,|\,u(x)=u(|x|), \,\forall x\in B\}=Fix_{O(2)}(H^2(B)\cap H^1_0(B)),$$
we deduce the existence of a critical point of $J_\sigma$ restricted to $H_{rad}(B)$. Then it is enough to notice that $J_\sigma$ is invariant under the action of $O(2)$ and to apply the Principle of Symmetric Criticality due to Palais (see \cite[Theorem 1.28]{PrSyCr}), retrieving that these points are critical for $J_\sigma$ also with respect to the whole space.\\
\noindent Finally, if we restrict to the interval $(-1,1]$, the positivity of such critical points is proved as in Propositions \ref{positivitysublinear} and \ref{positivitysuperlinear}, realizing that the superharmonic function of a radially symmetric function is radial too (cf. \eqref{tildePb}). On the other hand, if $\sigma>1$, one can apply the dual cone decomposition to the Hilbert space $H_{rad}(B)$ and argue as in Lemma \ref{dualcone} and Proposition \ref{possigma>1}, taking into account that $B$ is a [PPP$_\sigma$] domain for every $\sigma>-1$. Summarizing, we have shown the following:

\begin{prop}
	Let $p\in(0,1)\cup(1,+\infty)$, $g=g(|x|)\in L^1(B)$, $g>0$. If $\sigma\leq-1$, there is no positive nonnegative nontrivial solution for \eqref{PDEpSteklovRAD}, while, if $\sigma>-1$, there exists at least a positive radial solution, which is strictly superharmonic whenever $\sigma\in(-1,1]$.
\end{prop}

Now, we want to prove some qualitative properties of radial positive solutions of \eqref{PDEpSteklovRAD}. The first result concerns the radial behaviour, while the second the uniform boundedness in $L^\infty(B)$. Before proving these results, one should notice that such solutions are strong solutions, namely in $W^{4,q}(B)$, provided $g\in L^q(B)$ for some $q>2$ and also classical assuming in addition that $g\in W^{1,q}(B)$ for some $q>2$. This is a straightforward application of Lemma \ref{GGSestimate} combined with Morrey's embeddings.

\begin{lem}\label{maxprlem}
	Let $B:=B_R(0)$ be the ball of radius $R$ in $\R^2$ centered in $0$, $q>2$ and  $\tilde h\in W^{2,q}(B)$ be radial. Defining $h:[0,R]\rightarrow\R$ its restriction to the radial variable, for all $t\in[0,R]$ the following equality holds:
	\begin{equation}\label{maxpr}
	th'(t)=\int_0^ts\Delta h(s)ds.
	\end{equation}
\end{lem}
\begin{proof}
	If $h$ is of class $C^2$, it comes directly from integration by parts and from the radial representation of the laplacian as
	$$\Delta \tilde h(x)=h''(|x|)+\dfrac{1}{|x|}h'(|x|).$$
	Otherwise, let $(\tilde f_k)_{k\in\N}\subset C^\infty(\overline B)$ be such that $\tilde f_k\rightarrow \tilde h$ in $W^{2,q}(B)$, so in $C^1(\overline B)$. Since $\tilde h$ is radial, we claim that it is possible to choose each $\tilde f_k$ to be radial and we denote its restriction to the radial variable as $f_k$. If so, for every $k\in\N$, we have:
	$$tf_k'(t)=\int_0^ts\Delta f_k(s)ds.$$
	As a result, as $k\rightarrow+\infty:$
	$$\bigg|\int_0^ts(\Delta f_k(s)- \Delta h(s))ds\bigg|=\dfrac{1}{2\pi}\|\Delta \tilde f_k-\Delta \tilde h\|_{L^1(B_t(0))}\leq C(q)\|\tilde f_k-\tilde h\|_{W^{2,q}(B)}\rightarrow 0.$$
	The result is proved by the convergence in $C^1(\overline B)$ and the uniqueness of the limit. Now we have to justify our previous claim. Since $\tilde h\in W^{2,q}(B)$, we have $$\sum_{i,\alpha}\int_B\bigg|\dfrac{\partial^\alpha \tilde h}{\partial i^\alpha}(x,y)\bigg|^qdxdy<+\infty,$$
	where $i\in\{x,y\}$ and $\alpha$ is a multi-index of length $0\leq|\alpha|\leq 2$. Since each $\frac{\partial^\alpha \tilde h}{\partial i^\alpha}$ is radial, this is equivalent to say that 
	$h\in W^{2,q}([0,R],r)$, that is the weighted Sobolev space with weight $r$. Hence, by \cite[Theorem 7.4]{Kufner} ($M=\{0\}$, $\varepsilon=1$ in notation therein), there exists a sequence $(f_k)_{k\in\N}\subset C^\infty([0,R])$ such that $f_k\rightarrow h$ in $W^{2,q}([0,R],r)$, that is
	$$\sum_{i,\alpha}\int_0^Rr\bigg|\dfrac{\partial^\alpha h}{\partial i^\alpha}(r)-f_k(r)\bigg|^qdr\rightarrow0.$$
	Hence, defining $F_k(x):=f_k(|x|)$, then each $F_k\in C^\infty(\overline B)$, is radial and
	$$\|\tilde h-F_k\|_{W^{2,q}(B)}=\sum_{i,\alpha}\int_B\bigg|\dfrac{\partial^\alpha \tilde h}{\partial i^\alpha}(x,y)-F_k(x,y)\bigg|^qdxdy=2\pi\sum_{i,\alpha}\int_0^Rr\bigg|\dfrac{\partial^\alpha h}{\partial i^\alpha}(r)-f_k(r)\bigg|^qdr\rightarrow0,$$
	and the claim is proved.	
\end{proof}

\begin{prop}[Radial Decay]\label{decreasing}
	Assume $g\in \Lq(B)$ for some $q>2$, radial and $g>0$. and let $u\not\equiv 0$ be a nonnegative radial solution of \eqref{PDEpSteklovRAD} with  $\sigma\in(-1,1]$ and $p\in(0,1)\cup(1,+\infty)$. Then $u$ is strictly radially decreasing, thus $u>0$ in $B$.
\end{prop}
\begin{proof}
	By the assumption on $g$, we infer that $u$ is a strong solution, thus $w:=\Delta u\in W^{2,q}(B)$. Since $\Delta w=\Delta^2u=g(|x|)u^p\geq 0$ in $[0,1]$, applying Lemma $\ref{maxprlem}$, we have $w'>0$ in $(0,1]$. Hence $\Delta u$ is strictly increasing in $(0,1]$. Moreover, since $u$ is nonnegative and $u(1)=0$, we have $u'(1)\leq 0$; hence, using the second boundary condition, $\Delta u(1)=(1-\sigma)u'(1)\leq 0$. Since $\Delta u$ was strictly increasing in $(0,1]$, we deduce that $\Delta u<0$ in $[0,1)$, and finally, applying again Lemma $\ref{maxprlem}$, $u'<0$ in $(0,1]$.
\end{proof}

In the next result we find a uniform upper bound for positive radial solutions of \eqref{PDEpSteklovRAD}, which may be seen as a superlinear counterpart of Proposition \ref{upperboundH^2}. We will make use of a blow up method which goes back to Gidas and Spruck, \cite{GiSpr}, and which was adapted to the polyharmonic case by Reichel and Weth in \cite{RW,RW2}. Briefly, our argument will be the following: supposing the existence of a sequence of positive radial solutions with diverging $L^\infty$ norm, we rescale each of them in order to have an other sequence of functions with the same $L^\infty$ norm, satisfying the same equation in nested domains which tend to occupy the whole $\R^2$. Then we show that, up to a subsequence, it converges uniformly on compact subsets to a continuous nonnegative but nontrivial function. This turns out to be a solution of the same equation on $\R^2$, which is a contradiction with the following Liouville-type result by Wei and Xu, with $N=2$ and $m=2$:

\begin{lem}(\cite{WeiXu}, Theorem 1.4)\label{weixu}
	Let $m\in\N $ and assume that $p>1$ if $N\leq 2m$ and $1<p\leq\frac{N+2m}{N-2m}$ if $N>2m$. If $u$ is a classical nonnegative solution of $$(-\Delta)^m u=u^p\,\mbox{ in }\,\R^N,$$ then $u\equiv 0$.
\end{lem}

\noindent In our proof we will make also use of the following local regularity estimate, which is a particular case of a more general result by Reichel and Weth:

\begin{lem}(\cite{RW}, Corollary 6)\label{RWthm}
	Let $\Omega=B_R(0)\subset\R^N$, $m\in\N$, $h\in\Lp(\Omega)$ for some $p\in(1,+\infty)$ and suppose $u\in W^{2m,p}(\Omega)$ satisfies
	$$(-\Delta)^m u=h\,\,\mbox{ in }\,\Omega,$$
	then there exists a constant $C=C(R,N,p,m)$, such that for any $\delta\in(0,1)$,
	\begin{equation*}
	\|u\|_{W^{2m,p}(B_{\delta R}(0))}\leq\dfrac{C}{(1-\delta)^{2m}}(\|h\|_{\Lp(B_R(0))}+\|u\|_{\Lp(B_R(0))}).
	\end{equation*}
\end{lem}

\begin{prop}\label{uniformupperbound}
	Let $\sigma\in (-1,1]$ and $g\in\Lq(B)$ for some $q>2$, radial and $g>0$. Suppose also that $g$ is continuous in 0. Then, there exists $C>0$ independent of $\sigma$ such that $\|u\|_{\infty} \leq C$ for every $u$ radial positive solution of \eqref{PDEpSteklovRAD}.
\end{prop}
\begin{proof}
	By contradiction, suppose there exists a sequence $(v_k)_{k\in\N}$ of radial positive solutions such that $\|v_k\|_{\infty}\nearrow+\infty$. According to Proposition \ref{decreasing}, each $v_k$ is radially decreasing, so $v_k(0)=\|v_k\|_{\infty}\nearrow+\infty$. For each $k\geq 1$, define
	$$u_k(x)=\lambda_k^{\frac{4}{p-1}}v_k(\lambda_kx),$$
	where $\lambda_k\in\R^+$ are such that $\lambda_k^{\frac{4}{p-1}}=1/v_k(0)$. 
	With this choice, each $u_k$ satisfies
	\begin{equation}\label{PDEpSteklovRESCALED}
	\begin{cases}
	\Delta^2u_k=g(|\lambda_k x|)u_k^p\quad&\mbox{in }B_{\frac{1}{\lambda_k}}(0)\\
	u_k=\Delta u_k-(1-\sigma)\lambda_k(u_k)_n=0\quad&\mbox{on }\partial B_{\frac{1}{\lambda_k}}(0),
	\end{cases}
	\end{equation}
	is in $W^{4,q}(B_{\frac{1}{\lambda_k}}(0))$, radially decreasing and
	\begin{equation}\label{unif1}
	\|u_k\|_{L^\infty(B_{\frac{1}{\lambda_k}}(0))}=u_k(0)=\lambda_k^{\frac{4}{p-1}}v_k(0)=1,
	\end{equation}
	\noindent We claim that the sequence $(u_k)_{k\in\N}$ is uniformly bounded on compact sets of $\R^2$ in $W^{4,q}$ norm. In fact, let $K\subset\R^2$ be compact, then there exists $\rho>0$ such that $B_\rho(0)\supset K$ and, for $k$ large enough, each $u_k$ is well defined in $K$ since $B_{\frac{1}{\lambda_k}}(0)\supset B_{2\rho}(0)$ definitively. For such $k$, by \eqref{unif1} and applying Lemma \ref{RWthm} with $\Omega=B_{2\rho}(0)$, $m=N=2$ and $\delta=1/2$,
	\begin{equation}\label{stimaunifblowup}
	\begin{split}
	\|u_k\|_{W^{4,q}(K)}&\leq\|u_k\|_{W^{4,q}(B_\rho(0))}\leq\dfrac{C(\rho,q)}{(1/2^4)}(\|\Delta^2u_k\|_{\Lq(B_{2\rho}(0))}+\|u_k\|_{\Lq(B_{2\rho}(0))})\\
	&\leq 16C(\rho,q)(\|g(|\lambda_k \cdot|)\|_{L^q(B_{2\rho}(0))}+|B_{2\rho}(0)|^{\frac{1}{q}}).
	\end{split}
	\end{equation}
	Moreover, fixing $\varepsilon>0$ and supposing $k$ large enough,
	\begin{equation}\label{stimag}
	\|g(|\lambda_k \cdot|)\|_{L^q(B_{2\rho}(0))}=(4\pi\rho^2)^{\frac1q}\bigg(\dfrac1{|B_{2\rho\lambda_k}(0)|}\int_{B_{2\rho\lambda_k}(0)}|g(y)|^qdy\bigg)^{\frac1q}\leq(4\pi\rho^2)^{\frac1q}g(0)+\varepsilon
	\end{equation}
	where the last inequality follows from the Lebesgue Differentiation Theorem. Hence, combining \eqref{stimaunifblowup} with \eqref{stimag}, we infer $\|u_k\|_{W^{4,q}(K)}\leq C(p,q,K,g)$, so uniform on $k$. Incidentally, notice that this constant does not depend on $\sigma$. Hence we find $u\in W^{4,q}(K)$ such that, up to subsequences, $u_k\rightarrow u$ in $C^3(K)$, where $u\in C^3(\R^2)$, $u\geq0$ and $u(0)=1$ by \eqref{unif1} and
	satisfying 
	$$\Delta^2 u=g(0)u^p\quad\mbox{in}\,\,\R^2.$$
	so, by a bootstrap method, we deduce that $u$ is also a classical solution.
	Finally, setting for all $x\in\R^2$ $w(x):=u(bx)$ with $b:=g(0)^{-1/4}$, one has $w$ is a nonnegative solution of
	$$\Delta^2 w=w^p\quad\mbox{in}\,\,\R^2,$$
	with $w(0)=u(0)=1$, which contradicts Lemma \ref{weixu}.
\end{proof}

\subsection{Convergence results}
We want to investigate what happens at the endpoints of the interval $(-1,1]$ in which $\sigma$ lies, by means of the last results. More precisely, our aim is to examine if any result similar to Theorems \ref{approachingsigma*} and \ref{convergence} can be found assuming $(u_k)_{k\in\N}$ to be a sequence of positive radial solutions of \eqref{PDEpSteklovRAD} with $\sigma=\sigma_k$ but without imposing any "minimizing" requirement. Unless otherwise stated, we assume $g\equiv 1$ and $p>1$.
\vskip0.2truecm
\noindent Let us start with the behaviour for $\sigma\rightarrow1$, where the main ideas are taken from the same result for ground states. Notice that we know everything for the Navier problem in the ball: in fact, Dalmasso proved in \cite{Dalmasso} that there exists a unique positive solution, which is radially symmetric and radially decreasing thanks to a result by Troy, \cite{Troy}.

\begin{prop}\label{convergenceradposNAV}
	Let $(u_k)_{k\in\N}$ be a sequence of positive radial solutions of \eqref{PDEpSteklovRAD} with $\sigma_k\nearrow1$. Then $u_k\rightarrow\overline u$ in $H^2(B)$, where $\overline{u}$ is the unique positive solution of the Navier problem.
\end{prop}
\begin{proof}
	We firstly claim that such a sequence is bounded in $H^2(B)$. Indeed, by Proposition \ref{uniformupperbound}:
	\begin{equation*}
	\|u_k\|_{H^2(B)}^2\leq C_0\|\Delta u_k\|_2^2\leq C_0\bigg(1-\dfrac{1-\sigma_k}{2}\bigg)^{-1}\|u_k\|_{H_{\sigma_k}}^2=\dfrac{2C_0}{1+\sigma_k}\|u_k\|_{p+1}^{p+1}\leq 2\pi C_0C^{p+1}.
	\end{equation*}
	Hence, we can extract a subsequence $(u_{k_j})_{j\in\N}$ such that there exists $\overline{v}\in H^2(B)\cap H^1_0(B)$ such that $u_{k_j}\rightharpoonup\overline v$ weakly in $H^2(B)$. By Proposition \ref{fromweaktostrong}, together with Remark \ref{fromweaktostrongremark}, one can infer that this subsequence is actually strongly convergent in $H^2(B)$ and then that $\overline v$ is a weak solution of the Navier problem (thus classical by regularity theory). Moreover, since the convergence is pointwise, we immediately deduce that $\overline v$ is nonnegative, radially symmetric and radially non-increasing. Nevertheless, by Proposition \ref{decreasing}, $\overline{v}$ is actually strictly decreasing and positive in $B$, so it coincides with the unique positive solution $\overline{u}$ of the Navier problem. By the uniqueness of the limit and applying Urysohn subsequence principle, we retrieve the convergence of the whole sequence $(u_k)_{k\in\N}$ from which we started.
\end{proof}

Let us now investigate the case $\sigma\rightarrow-1$. As already noticed in Lemma \ref{approachingsigmastar}, it is enough to understand the behaviour of the $L^{p+1}(B)$ norm of a sequence of solutions to infer the convergence in $H^2(B)$ norm. Since the proof of Theorem \ref{approachingsigma*} strongly relies on the fact that it deals with ground states, we need a different technique. The first step is a Pohozaev-type identity by Mitidieri in \cite{Mitidieri}: it will allow us to prove an inequality involving $L^p(B)$ and $L^{p+1}(B)$ norms which, combined with the uniform bound of Proposition \ref{uniformupperbound}, will lead us to the convergence result.

\begin{lem}[\cite{Mitidieri}, Proposition 2.2]\label{mitidieri}
	Let $\Omega$ be a smooth domain and $u\in C^4(\Omegabar)$. The following identity holds:
	\begin{equation*}
	\begin{split}
	&\int_\Omega(\Delta^2u)x\cdot\nabla u -\dfrac{N}{2}\int_\Omega(\Delta u)^2 -(N-2)\int_\Omega \nabla\Delta u\cdot\nabla u=-\dfrac{1}{2}\int_\dOmega(\Delta u)^2x\cdot n\\
	&+\int_\dOmega\bigg((\Delta u)_n(x\cdot\nabla u)+u_n(x\cdot\nabla\Delta u)-\nabla\Delta u\cdot\nabla u(x\cdot n)\bigg).
	\end{split}
	\end{equation*}
\end{lem}

\begin{cor}\label{PhozaevMitBGWlem}
	Suppose $u$ is a positive solution for problem \eqref{PDEpSteklovRAD} with $g\equiv1$, then the following identity holds:
	\begin{equation}\label{PhozaevMitBGW}
	\int_{\partial B_R}\bigg((\Delta u)_n+(1-\sigma)\bigg(1-\dfrac{1-\sigma}{2}\bigg)u_n\bigg)u_n=-\bigg(1+\dfrac{2}{p+1}\bigg)\int_{B_R} u^{p+1}.
	\end{equation}
\end{cor}
\begin{proof}
	By similar computations as in the proof of Section 6 of \cite{BGW}, from Lemma \ref{mitidieri} one infers:
	\begin{equation}\label{mitquasifinita}
	\bigg(\dfrac{N-4}{2}-\dfrac{N}{p+1}\bigg)\int_\Omega u^{p+1}=\int_\dOmega\bigg(x\cdot\nabla\Delta u+\dfrac{N}{2}(1-\sigma)\kappa u_n-\dfrac{1}{2}(1-\sigma)^2\kappa^2u_n(x\cdot n)\bigg)u_n.
	\end{equation}
	If $N=2$, $\Omega=B$, we have $x=n$ and $\kappa=1$, so $x\cdot\nabla\Delta u=(\Delta u)_n$ and \eqref{PhozaevMitBGW} follows.
\end{proof}

\noindent The next result follows some ideas of Berchio and Gazzola: we give here a sketch, while we refer to \cite[Proposition 4]{BG}, for a more detailed proof.
\begin{lem}\label{somelowerbound}
	Let $\sigma\in(-1,1)$ and $u$ be a positive radial solution of problem \eqref{PDEpSteklovRAD} with $g\equiv1$. Then the following estimate holds:
	\begin{equation}\label{stimaconcave}	\|u\|_{p+1}^{p+1}\geq\dfrac{3}{64}\bigg(1-\dfrac{3}{64}(1-\sigma)\bigg)\dfrac{1}{\pi(1+\sigma)}\dfrac{p+1}{p+3}\|\Delta^2 u\|_1^2.
	\end{equation}
\end{lem}
\begin{proof}
	By radial symmetry, \eqref{mitquasifinita} reduces to
	\begin{equation}\label{PhozaevMitBGWmodified}
	2(\Delta u)'(1)u'(1)+(1-\sigma)(1+\sigma)(u'(1))^2=-\dfrac{p+3}{p+1}\dfrac{1}{\pi}\int_B u^{p+1}
	\end{equation}
	Moreover, by Divergence Theorem we have
	$$u'(1)=\dfrac{1}{2\pi}\int_B\Delta u\quad\,\mbox{and}\quad\,(\Delta u)'(1)=\dfrac{1}{2\pi}\int_B\Delta^2 u,$$
	so, taking the first Steklov eigenfunction $w(x)=\frac{1}{4}(1-|x|^2)$ and after some elementary computations, one gets:
	\begin{equation}\label{LHSRHS}
	\bigg(\int_B\Delta^2 u-(1-\sigma)\int_B w\Delta^2 u\bigg)\int_B w\Delta^2 u=\dfrac{p+3}{p+1}(1+\sigma)\pi\int_B u^{p+1}.
	\end{equation}
	Noticing that $0\leq w\leq 1/4$, we have
	$$\dfrac{3}{64}\int_B \Delta^2 u\leq \int_B w\Delta^2u \leq\dfrac{1}{4}\int_B \Delta^2 u.$$
	Hence, defining now $d:=(1-\sigma)$, $s:=\int_B w\Delta^2u$ and $A:=\int_B \Delta^2 u$, LHS of \eqref{LHSRHS} becomes
	$$As-ds^2\quad\quad\mbox{with }s\in\bigg[\dfrac{3}{64}A;\dfrac{1}{4}A\bigg].$$
	Since $d>0$, $\psi:s\mapsto As-ds^2$ is a concave function, so it attains its minimum on the extremal values of the interval: in this case, with $0<d<2$, one has
	$$\psi(s)\geq\dfrac{3}{64}\bigg(1-\dfrac{3}{64}d\bigg)A^2.$$
	Combining this with \eqref{LHSRHS}, one finds the desired estimate \eqref{stimaconcave}.	
\end{proof}

\begin{thm}
	Let $\sigma_k\searrow-1$ and $(u_k)_{k\in\N}$ be a sequence of positive radial functions, each of them solution of the problem \eqref{PDEpSteklovRAD} with $g\equiv1$ and $\sigma=\sigma_k$. Then, $u_k\rightarrow 0$ in $H^2(B)$.
\end{thm}
\begin{proof}
	By Lemma \ref{approachingsigmastar}, it is enough to prove the convergence in $L^{p+1}(B)$ norm. Since every solution of \eqref{PDEpSteklovRAD} is smooth, we have $\|\Delta^2 u_k\|_1=\|u_k\|_p^p$. Moreover, by the uniform $L^\infty$ estimate found in Proposition \ref{uniformupperbound}, we know that there exists a constant $C>0$ not depending on $\sigma_k$, such that
	$$\|u_k\|_{p+1}^{p+1}\leq\|u_k\|_\infty^{p+1}|B|\leq\pi C^{p+1}.$$
	As a result, using the estimate provided by Lemma \ref{somelowerbound}, one has
	\begin{equation}\label{estimate}
	\dfrac{1+\sigma_k}{1-\frac{3}{64}(1-\sigma_k)}\geq\dfrac{p+1}{p+3}\dfrac{3}{64\pi^2 C^{p+1}}\|u_k\|_p^{2p},
	\end{equation}
	so, letting $\sigma_k\rightarrow -1$ we deduce $\|u_k\|_p\rightarrow 0$. This, together with the $L^\infty(B)$ estimate of Proposition \ref{uniformupperbound}, gives us the convergence in $L^{p+1}(B)$ and so the desired result.
\end{proof}

\section{Open problems}
We end our paper with some unsolved questions that would complete the present investigation.
\begin{itemize}
	\item \textit{If $\Omega$ is a ball, are the ground states of $J_\sigma$ radially symmetric?}
\end{itemize}
In fact, we deduced the existence of ground states and radial solutions which are indeed ground states among all possible radial solutions; both of them are positive and have the same behaviour when $\sigma\rightarrow -1$ and $\sigma\rightarrow 1$. But no standard techniques such as Talenti symmetrization principle seem to apply (except for the Navier case) to prove that these classes of functions are indeed the same.
\begin{itemize}
	\item \textit{Are the radial positive solutions radially decreasing if $\sigma>1$?}
\end{itemize}
Indeed, the radial decay property proved in Proposition \ref{decreasing} does not apply in this setting and, by now, we cannot extend Proposition \ref{uniformupperbound} for these values of $\sigma$.
\vskip0.2truecm
\noindent Moreover, in the spirit of \cite{Dalmasso} and \cite{FGW},
\begin{itemize}
	\item \textit{can we say something about the uniqueness of (at least) the positive radially symmetric ground state of $J_\sigma$, for some values of $\sigma$?}
\end{itemize}
\vskip0.2truecm
Finally, all the techniques developed from Section 3 strongly relied on the assumptions we made on the boundary, that is $\dOmega$ of class $C^{1,1}$ in order to have $\kappa\in L^\infty(\dOmega)$. In particular, Theorem \ref{detbordo2} allowed us to rewrite in an appropriate way our functional. Also the convexity played a crucial role to prove the positivity: see in particular Propositions \ref{positivitysublinear}, \ref{positivitysuperlinear}, \ref{extension1} as well as Theorem \ref{possigma>1thm}.
\begin{itemize}
	\item \textit{May we deduce the positivity of ground states of $J_\sigma$ when the domain $\Omega$ is not convex anymore or with less regularity on the boundary?}
\end{itemize}
Since in the Navier case their positivity is always assured simply by the maximum principle, we expect that, even without the convexity assumption, it continues to hold whenever $\sigma$ belongs to a neighborhood of $1$ which may depend on "how far" the domain is from being convex.\\
\noindent Concerning the regularity of the boundary, if we consider the particular case of a convex polygon $\mathcal P$, it is known that ground states of $J_\sigma$ are positive for every $\sigma$: in fact, the superharmonic method applies easily once we have $\int_{\mathcal P}\det(\nabla^2u)=0$ thanks to a result by Grisvard \cite[Lemma 2.2.2]{Grisvard}. We believe that positivity for ground states of $J_\sigma$ still holds imposing, for instance, only Lipschitz regularity for $\dOmega$.


\bigskip

\begin{thebibliography}{47}
	
	
	\bibitem{A} Adams, R., Fournier, J.
	{\em Sobolev spaces}.  Second edition. Pure and Applied Mathematics (Amsterdam), 140. Elsevier/Academic Press, Amsterdam, 2003. xiv+305 pp.
	
	\bibitem{Adolfsson} Adolfsson, V.
	{\em $L^2$-integrability of second-order derivatives for Poisson's equation in nonsmooth domains.} Math. Scand. 70 (1992), no. 1, 146-160.
	
	\bibitem{ADN} Agmon, S., Douglis, A., Nirenberg, L.
	{\em Estimates near the boundary for solutions of elliptic partial differential equations satisfying general boundary conditions.} I. Comm. Pure Appl. Math., 12:623-727, 1959.
	
	\bibitem{AM} Ambrosetti, A., Malchiodi, A.
	{\em Nonlinear analysis and semilinear elliptic problems.} Cambridge Studies in Advanced Mathematics, 104. Cambridge University Press, Cambridge, 2007. xii+316 pp.
	
	\bibitem{BG} Berchio, E., Gazzola, F.
	{\em Positive solutions to a linearly perturbed critical growth biharmonic problem.} Discrete Contin. Dyn. Syst. Ser. S 4 (2011), no. 4, 809-823.
	
	\bibitem{BGM} Berchio, E., Gazzola, F., Mitidieri, E.
	{\em Positivity preserving property for a class of biharmonic elliptic problems.} J. Differential Equations 229 (2006), 1-23.
	
	\bibitem{BGW} Berchio, E., Gazzola, F., Weth, T.
	{\em Critical growth biharmonic elliptic problems under Steklov-type boundary conditions.} Adv. Differential Equations 12 (2007), no. 4, 381-406. 
	
	\bibitem{BM} Brezis, H., Mironescu, P. 
	{\em Gagliardo-Nirenberg, composition and products in fractional Sobolev spaces.} J. Evol. Equ. 1 (2001), no. 4, 387-404.
	
	\bibitem{BucurFG} Bucur, D., Ferrero, A., Gazzola, F. 
	{\em On the first eigenvalue of a fourth order Steklov problem.} Calc. Var. Partial Differential Equations, 35:103-131, 2009.

	\bibitem{CCN} Castro, A., Cossio, J., Neuberger, J.M.
	{\em A sign-changing solution for a superlinear Dirichlet problem.} Rocky Mountain J. Math. 27 (1997), no. 4, 1041-1053. 
	
	\bibitem{Dalmasso} Dalmasso, R.
	{\em Uniqueness theorems for some fourth-order elliptic equations.} Proc. Amer. Math. Soc. 123 (1995), no. 4, 1177-1183. 

	\bibitem{E} Evans, L.C.
	{\em Partial Differential Equations.} Second edition. Graduate Studies in Mathematics, 19. American Mathematical Society, Providence, RI, 2010. xxii+749 pp.

	\bibitem{FGW} Ferrero, A., Gazzola, F., Weth, T.
	{\em Positivity, symmetry and uniqueness for minimizers of second-order Sobolev inequalities.} Ann. Mat. Pura Appl. (4) 186 (2007), no. 4, 565-578.
	
	\bibitem{FGW2} Ferrero, A., Gazzola, F., Weth, T.
	{\em On a fourth order Steklov eigenvalue problem.} Analysis (Munich) 25 (2005), no. 4, 315-332.
	
	\bibitem{GP} Gasinski, L., Papageorgiou, N.S.
	{\em Nonlinear analysis.} Series in Mathematical Analysis and Applications, 9. Chapman and Hall/CRC, Boca Raton, FL, 2006. xii+971 pp.
	
	\bibitem{GGr} Gazzola, F., Grunau, H.-Ch.
	{\em Critical dimensions and higher order Sobolev inequalities with remainder terms.} NoDEA Nonlinear Differential Equations Appl. 8 (2001), no. 1, 35-44.
	
	\bibitem{GGS} Gazzola, F., Grunau, H.-Ch., Sweers, G.
	{\em Polyharmonic boundary value problems.} Springer Lecture Notes in Mathematics n. 1991, 2010. xviii+423 pp.
	
	\bibitem{GS} Gazzola, F., Sweers, G.
	{\em On positivity for the biharmonic operator under Steklov boundary conditions.}  Arch. Ration. Mech. Anal. 188 (2008), no. 3, 399-427. 
	
	\bibitem{GiSpr} Gidas, B., Spruck, J.
	{\em A priori bounds for positive solutions of nonlinear elliptic equations.} Comm. Partial Differential Equations 6 (1981), no. 8, 883-901.
	
	\bibitem{GT} Gilbarg, D., Trudinger, N.S.
	{\em Elliptic Partial Differential Equations of Second Order}. Third edition. Classics in Mathematics, Springer-Verlag, Berlin, 2001. xiv+517 pp.
	
	\bibitem{Grisvard} Grisvard, P.
	{\em Singularities in boundary value problems.} Recherches en Math\'{e}matiques Appliqu\'{e}es [Research in Applied Mathematics], 22. Masson, Paris; Springer-Verlag, Berlin, 1992. xiv+199 pp.
	
	\bibitem{PG} Grumiau, C., Parini, E.
	{\em On the asymptotics of solutions of the Lane-Emden problem for the p-Laplacian.} Arch. Math. 91 (2008). 354-365.
	
	\bibitem{Kufner} Kufner, A.
	{\em Weighted Sobolev spaces.} Translated from the Czech. A Wiley-Interscience Publication. John Wiley and Sons, Inc., New York, 1985. 116 pp.
	
	\bibitem{Marsch} Marschall, J.
	{\em The trace of Sobolev-Slobodeckij spaces on Lipschitz domains.} Manuscripta Math. 58 (1987), no. 1-2, 47-65.
	
	\bibitem{Mitidieri} Mitidieri, E.
	{\em A Rellich type identity and applications.} Comm. Partial Differential Equations 18 (1993), no. 1-2, 125-151.
	
	\bibitem{Moreau} Moreau, J.J.
	{\em D\'ecomposition orthogonale d'un espace hilbertien selon deux c\^ones mutuellement polaires.} (French) C. R. Acad. Sci. Paris 255 1962 238-240.
	
	\bibitem{NStyS} Nazarov, S.A., Stylianou, A., Sweers, G.
	{\em Hinged and supported plates with corners.} Z. Angew. Math. Phys. 63 (2012), no. 5, 929-960.

	\bibitem{PS} Parini, E., Stylianou, A.
	{\em On the positivity preserving property of hinged plates.} SIAM J. Math. Anal. 41 (2009), no. 5, 2031-2037. 
	
	\bibitem{RW} Reichel, W., Weth, T.
	{\em A priori bounds and a Liouville theorem on a half-space for higher-order elliptic Dirichlet problems},  Math. Z. 261 (2009), no. 4, 805-827
	
	\bibitem{RW2} Reichel, W., Weth, T.
	{\em Existence of solutions to nonlinear, subcritical higher order elliptic Dirichlet problems.} J. Differential Equations 248 (2010), no. 7, 1866-1878
	
	\bibitem{Sperb} Sperb, R.P.
	{\em Maximum Principles and Their Applications}. Mathematics in Science and Engineering, 157. Academic Press, New York-London, 1981. ix+224 pp.
	
	\bibitem{Stek} Stekloff, W.
	{\em Sur les probl\`{e}mes fondamentaux de la physique math\'ematique}. (French) Ann. Sci. \'{E}cole Norm. Sup. (3) 19 (1902), 191-259.
	
	\bibitem{Sty} Stylianou, A.
	{\em Comparison and sign preserving properties of bilaplace boundary value problems in domains with corners.} Ph.D. thesis, Verlag Dr. Hut - M\"{u}nchen, 2010.
	
	\bibitem{SwRew} Sweers, G.
	{\em When is the first eigenfunction for the clamped plate equation of fixed sign?} Proceedings of the USA-Chile Workshop on Nonlinear Analysis (Vi\~{n}a del Mar-Valparaiso, 2000), 285-296.
	
	\bibitem{Troy} Troy, W.C.
	{\em Symmetry properties in systems of semilinear elliptic equations.} J. Differ. Equ. 42, (1981), 400-413.
	
	\bibitem{VK} Ventsel, E., Krauthammer, T.
	{\em Thin plates and shells: theory: analysis, and applications.}  CRC press, 2001. 688 pp.
	
	\bibitem{WeiXu} Wei, J., Xu, X.
	{\em Classification of solutions of higher order conformally invariant equations.} Math. Ann. 313 (1999), no. 2, 207-228.
	
	\bibitem{PrSyCr} Willem, M.
	{\em Minimax theorems.} Progress in Nonlinear Differential Equations and their Applications, 24. Birkh\"{a}user Boston, Inc., Boston, MA, 1996. x+162 pp.
	
\end{thebibliography}
\end{document}